\documentclass{amsart}
\usepackage{amsmath,amsthm}
\usepackage{amsfonts,amssymb}
\usepackage{enumerate}
\RequirePackage{srcltx}

\numberwithin{equation}{section}
\newtheorem{thm}{Theorem}[section]
\newtheorem{cor}[thm]{Corollary}
\newtheorem{lem}[thm]{Lemma}

\newtheorem{exam}[thm]{Example}

\newtheorem{defn}[thm]{Definition}

\theoremstyle{remark}
\newtheorem{rem}[thm]{Remark}

\numberwithin{equation}{section}

\newcommand{\al}{\alpha}
\def\be{\beta}
\def\sph{\mathbb{S}^{d-1}}
\def \b{\beta}
\def\l{\left}

\def\Ld{\Lambda}
\def\da{\delta}

\def\og{\omega}

\def\Bl{\Bigl}
\def\Br{\Bigr}
\def\f{\frac}
\def\({\Bigl(}
\def \){ \Bigr)}

\def\Ga{\Gamma}

\def\hb{\hfill$\Box$}

\def\la{\label}
\newcommand{\ld}{\lambda}

\newcommand\p{\partial}

 \def\a{{\alpha}}
 \def\b{{\beta}}
 \def\g{{\gamma}}
 
 \def\t{{\theta}}
 \def\l{{\lambda}}
 \def\d{{\delta}}
 \def\o{{\omega}}
 \def\s{{\sigma}}
 \def\la{{\langle}}
 \def\ra{{\rangle}}
 \def\ve{{\varepsilon}}

 \def\CC{{\mathbb C}}
 
 \def\NN{{\mathbb N}}

 \def\RR{{\mathbb R}}
 \def\ZZ{{\mathbb Z}}
        
        \def\proj{\operatorname{proj}}

\def\sa{\sigma}

\def\ta{\theta}
\newcommand{\tr}{\triangle}

\def\bi{\bibitem}

\def\va{\varepsilon}
\def\vi{\varphi}

\def\RR{\mathbb{R}}

\def\SS{\mathbb{S}}

\def\ra{\rangle}

\def\la{\langle}

\def\HH{\mathcal{H}}

\begin{document}

\title[Weighted Fractional Bernstein's inequalities]
{Weighted Fractional Bernstein's inequalities 
\\ and their applications
}
\author{Feng Dai}
\address{Department of Mathematical and Statistical Sciences\\
University of Alberta\\ Edmonton, Alberta T6G 2G1, Canada.}
\email{fdai@ualberta.ca}
\author{Sergey Tikhonov} \address{ICREA and Centre de Recerca Matem\`{a}tica\\
Campus de Bellaterra, Edifici C
08193 Bellaterra (Barcelona), Spain.}
\email{ stikhonov@crm.cat}

\date{\today}

\thanks{The first  author  was partially supported  by the NSERC Canada
under grant RGPIN 311678-2010.  The second author was partially supported  by
 MTM 2011-27637, 2009 SGR 1303,  RFFI 13-01-00043, and NSH-979.2012.1.
}

\begin{abstract}

 This paper studies  the following  weighted,  fractional  Bernstein inequality for spherical polynomials on $\sph$:
  \begin{equation}\label{4-1-TD-ab}
\|(-\Delta_0)^{r/2} f\|_{p,w}\leq C_w n^{r} \|f\|_{p,w}, \   \  \forall f\in \Pi_n^d,
\end{equation}
 where $\Pi_n^d$ denotes the space of all spherical polynomials of degree at most $n$ on $\sph$, and $(-\Delta_0)^{r/2}$ is the fractional  Laplacian-Beltrami operator on $\sph$.
  A new class of doubling weights with conditions weaker than the $A_p$ is introduced, and used to  fully  characterize  those  doubling weights $w$ on $\sph$ for which the weighted Bernstein inequality \eqref{4-1-TD-ab} holds for some  $1\leq p\leq \infty$ and all $r>\tau$.  In the unweighted case, it is shown   that
 if  $0<p<\infty$ and $r>0$ is not an even integer, then \eqref{4-1-TD-ab} with $w\equiv 1$ holds if and only if $r>(d-1)(\f 1p-1)$. As applications, we show that any function $f\in L_p(\sph)$ with $0<p<1$ can be approximated by the de la Vall\'ee Poussin means of a Fourier-Laplace series, and establish a sharp  Sobolev type Embedding theorem for the weighted Besov spaces with respect to general doubling weights.

\end{abstract}

\maketitle

\section{Introduction}

 One of the  fundamental results in analysis is the following Bernstein inequality for trigonometric polynomials: 
\begin{equation} \label{ber}
\|f^{(r)}\|_{p} \le C {n^r}\|f\|_{p},\qquad 0<p\le \infty, \quad r\in\mathbb{N},\  \   f\in\mathcal{T}_n,
\end{equation}
 where $\|\cdot\|_p=\|\cdot\|_{L^p[0,2\pi]}$,  $\mathcal{T}_n$ denotes the space of all trigonometric polynomials of degree at most $n$, and $C=1$ is known to be   the best constant (see \cite[p. 16, (4.4)]{arestov}).
In  \cite[p. 45, Theorem 4.1]{MT2}, Mastroianni and Totik established  a weighted analogue of  (\ref{ber}) for all doubling weights. Among other things, they proved that for any doubling weight $w$,
\begin{equation} \label{berw}
\|f^{(r)}\|_{p,w} \le C_w {n^r}\|f\|_{p,w},\   \  \forall f\in\mathcal{T}_n, \   \ r\in \NN, \  \    1\leq p\leq \infty,
\end{equation}
where $\|f\|_{p,w}=\|f\,w^{1/p}\|_{p}$, and $C_w$ depends only on the doubling constant of $w$.
Later on, \eqref{berw} was extended to the case of  $0< p< 1$ by Erd\'{e}lyi   \cite[p. 69, Theorem 3.1]{er}.

For 
spherical polynomials  on the unit sphere $\sph$,
it was shown in \cite[Corollary 5.2, p. 155]{Dai1} that  if $r$ is an even  integer and $w$ is a doubling weight, then the weighted Bernstein inequality,
\begin{equation}\label{4-1-TD-new}
\|(-\Delta_0)^{r/2} f\|_{p,w}\leq C_w n^{r} \|f\|_{p,w}, \   \  \forall f\in \Pi_n^d,
\end{equation}
holds for all $0<p\leq \infty$, where $\Pi_n^d$ denotes the space of all spherical polynomials of degree at most $n$ on $\sph$, and $\Delta_0$ is the Laplacian-Beltrami operator on $\sph$.   In the unweighted case (i.e., $w=1$),
\eqref{4-1-TD-new} was shown earlier in \cite[p.330, Theorem 3.2]{Di3} for 
 all $1\leq p\leq \infty$.

 The fractional Bernstein inequality, namely, the inequality  \eqref{berw} or \eqref{4-1-TD-new} for   positive $r$ that may not be an integer,  plays an important role in harmonic analysis and PDE (see, for instance,  \cite{wu, wu1}), and the investigation of this   inequality has a long history.
Firstly, Lizorkin \cite{li} showed that (\ref{ber}) holds for all  $r>0$ and $1\leq p\leq\infty$. ( A similar result for functions of exponential type was also established in \cite{li}). Secondly, the fractional Bernstein inequality for trigonometric polynomials for $0<p<1$  was studied by Belinskii and   Liflyand \cite{BL}, who particularly observed that if $r>0$ is not an integer, then \eqref{ber} does not hold for the full range of $0<p<1$. Of related interest is the fact that the (unweighted) fractional  Bernstein inequality remains true   in  the  $H^p$ spaces  for all $0<p\leq 1 $ and  $r>0$.
 Finally, the fractional  Bernstein inequality with $1\leq p\leq \infty$  was established for   multivariate trigonometric polynomials, and for  spherical harmonics  in \cite{ru1, ru2} and \cite{Di3, kam}, respectively.

In this paper, we shall study the weighted,  fractional  Bernstein inequality for spherical polynomials on $\sph$ as well as its applications in approximation theory.
 We shall give a full  characterization of all those  doubling weights for which the weighted Bernstein inequality \eqref{4-1-TD-new} holds for some $r\notin 2\NN$ and $1\leq p\leq \infty$.
 It turns out that  there is a considerable difference
   between the cases of integer power and non-integer power
   (i,e.,  fractional power)  of the Laplace-Beltrami operator on the sphere.
   In fact,  in
 the unweighted case, we  prove the following.

\begin{thm} \label{sharp-bernstein}
If $0<p<\infty$,  $r>0$, and $d\ge 3$, then
\begin{align}
&\sup_{f\in \Pi_n^d, \  \|f\|_p\leq 1} \|(-\Delta_0)^{r/2} f\|_{L^p(\sph)}\notag\\
&\sim \begin{cases}
n^r, &\   \  \text{if $r>(d-1)(\f 1p-1)$ or $r\in 2\NN$;}\\
n^{r}\log^{\f 1p} n, &\    \ \text{if $r=(d-1)(\f 1p-1)$, and $r\notin 2\NN$};\\
n^{(d-1)(\f 1p-1)}, &\   \  \text{if $r<(d-1)(\f 1p-1)$, and $r\notin 2\NN$}.
\end{cases}\label{4-2-DT}
\end{align}
\end{thm}

According to  Theorem \ref{sharp-bernstein},
in the unweighted  case (i.e.,  $w=1$),
  the Bernstein inequality
\eqref{4-1-TD-new} for a non-integer (i.e., fractional)
 power  $ r/2$ of the Laplace-Beltrami operator
 holds if and only if $ p> \f{d-1}{d-1+r}$,
 whereas    \eqref{4-1-TD-new} for an integer  power $r/2$  holds for the full range of $0<p<\infty$.

 We point out that in the case when $d=2$ and $r$ is not an integer, Theorem \ref{sharp-bernstein}  is due to Belinskii and Liflyand \cite{BL}, where
 the proofs do  not  seem to work  for  the higher-dimensional case.

The paper is organized as follows.
Section 2 contains some preliminary results on
spherical polynomial expansions on the unit sphere, as well as   a technical theorem, Theorem \ref{prop-1-2-ch3}, which gives sharp
asymptotic  estimates of the weighted norms of certain kernel functions. This  theorem
  plays a crucial role in the proof of Theorem
   \ref{sharp-bernstein}, whereas  its  proof is
    postponed to the appendix.
Basic facts on doubling weights and several useful weighted polynomial inequalities   are presented  in Section 3.
The fourth  section  is devoted to the proof of the fractional Bernstein inequality for spherical polynomials on $\sph$.  Theorem \ref{sharp-bernstein},  as well as the weighted  Bernstein inequality with doubling weights for $0<p\leq \infty$  are proved in this section. After that, in Section 5, we show that our method can yield a better result for weighted  fractional Bernstein inequality
with the   Muckenhoupt $A_p$ weights.

One of our  main results in  this paper is given in Section 6,
 where we prove  a full  characterization of the doubling weights for which the weighted Bernstein inequality   holds. We introduce a new class $\mathbb{A}_{p,\tau}$ of  weights on $\sph$ and prove that the inequality (\ref{4-1-TD-new}) holds for any $r>\tau$ if and only if $w\in\mathbb{A}_{p,\tau}.$
In particular,
the inequality $\|f^{(r)}\|_{p,w} \le C_w {n^r}\|f\|_{p,w}$,  $1\leq p < \infty$, holds for a trigonometric
polynomial $f\in\mathcal{T}_n$ for any $r>\tau$ if and only if $w\in\mathbb{A}_{p,\tau}.$

In Section 7, we consider  spherical polynomial approximation  in $L^p$ for  $0<p<1$, 
 following  the approach  of Oswald for the trigonometric polynomials \cite{osw}.
 In particular, we show that
if $0<p<1$ and $f\in L^p(\sph)$, then
there exists a Fourier-Laplace series $\sa$ on the sphere $\sph$
  such that the following quantitative estimate holds:
$$
\|f-V_n\sa\|_p \leq C n^{-(d-1)(\f 1p-1)} \Bigl( \sum_{k=1}^n k^{d-2-(d-1)p} E_k (f)_p^p\Bigr)^{\f1p},
$$ where $V_n$ is the de la Vall\'ee Poussin operator, and $E_k(f)_p:=\inf_{g\in\Pi_n^d} \|f-g\|_p$.
If, in addition, $\sum_{k={n+1}}^\infty k^{(d-2)-(d-1)p}E_k(f)_p^p<\infty$, then $V_n f$ is well defined, and we have
$$ \|f-V_n f\|_p\leq C n^{-(d-1)(\f 1p-1)} \Bigl( \sum_{k={n+1}}^\infty k^{(d-2)-(d-1)p}E_k(f)_p^p\Bigr)^{\f1p}.
$$

In Section 8,  we show how to apply our result to deduce the  Sobolev-type embedding theorem for the weighted Besov spaces at the critical index.  
We prove that  if $0<p<q\le\infty$,
$w$ is a  doubling weight on $\sph$, and $\alpha=s_w(\frac1p-\frac1q)$,  then the weighted Besov space $B^\alpha_q(L_{p,w})$ can be continuously embedded into the space $L_{q,w}$, where $s_w$ is a geometric constant depending only on $w$. (The precise definition of $s_w$ is given in Section 3).  Examples will be given to show the index $\alpha=s_w(\frac1p-\frac1q)$, in general, is sharp.
This result 
 improves a result in \cite[Cor. 4]{hesse} and \cite[Th. 2.5]{DW}.
For the classical result, we refer to  the  paper of Peetre \cite[(8.2)]{peetre}.

Finally, we prove the technical result, Theorem \ref{prop-1-2-ch3}, in appendix.

\section{Preliminaries}

Let  $\sph=\{ x\in \mathbb{R}^{d}:\   \   \|x\|=1\}$ denote the
unit sphere of $\RR^d$ endowed  with the usual rotation-invariant
 measure $d\sa(x)$, where, and in what follows,  $\|x\|$
denotes the Euclidean norm of $x\in \RR^d$. Let  $\rho(x,y):=\arccos (x\cdot y)$ denote
 the usual geodesic distance of $x, y\in \sph$, and
   $ B(x,r):=\{ y\in\sph: \ \rho(x,y)\leq r\}$    the spherical cap
 centered   at $x\in\sph$  of  radius $r\in (0,\pi]$.
  Given  a constant $c>0$,   we use the notation
$cB:=B(x, cr)$  to denote the spherical cap with the same center as that of $B:=B(x,r)$ but
$c$ times the radius of $B$.
Given a set
$E\subset \sph$, we denote by
 $\chi_E$  and  $|E|$  the characteristic function of $E$ and the
Lebesgue measure $\sa(E)$ of  $E$, respectively. We shall use the
notation $ A\sim B$ to mean that there exists an inessential
constant $c>0$, called the constant of equivalence, such that
$$ c^{-1} A \leq B \leq c A.$$
For $0<p\leq \infty$ and $f\in L^p(\sph)$,  we define
$$ E_n(f)_{p} =\inf_{g\in\Pi_n} \|f-g\|_{p},\    \
n=0,1,2,\ldots.$$

A spherical polynomial of degree at most $n$ on $\sph$ is
the restriction to $\sph$ of a polynomial in $d$ variables of total
degree at most $n$. We denote by $ \Pi_n^d$ the space of all  real
spherical polynomials of degree at most $n$ on $\sph$. It is a finite dimensional vector space over $\RR$ with
 $\dim \Pi_n^d\sim  n^{d-1}.$

 Let $\HH_0^d$ denote the space of constant functions on $\sph$.
   For each positive integer $n$, we denote by  $\HH_n^d$  the
orthogonal complement of $\Pi_{n-1}^d$ in  $\Pi_n^d$
with respect to the inner product of $L^2(\sph)$.  $\HH_n^d$ is called the space of spherical harmonics of degree $n$ on $\sph$. Thus,
the spaces $\HH_n^d$, $n=0,1,\cdots$ of spherical harmonics are mutually orthogonal with respect to the inner product of $L^2(\sph)$, and for each $n\in\NN$,   $
\mathrm{dim}\, \mathcal{H}_n^d=\text{dim}\, \Pi_n^d-\text{dim}\Pi_{n-1}^d
\sim  n^{d-2}$. Since the space of spherical polynomials is dense in $L^2(\sph)$,    each $f\in L^2(\sph)$ has a spherical harmonic  expansion: \begin{equation}\label{s-expansions}
f=\sum_{k=0}^\infty \proj_k f,\end{equation}
   where $\proj_k$ is the orthogonal projection of $L^2(\sph)$ onto the space $\HH_k^d$ of spherical harmonics, which has an integral representation: \begin{equation}\label{2-2} \proj_kf(x) =\f { \Ga (\f {d-1}2)}{ \Ga (d-1) |\sph|}\int_{\sph} f(y) E_k^{(\f {d-3}2, \f
{d-3}2)}(x\cdot y)\, d\sa(y),\ \  \ x\in\sph. \end{equation}
Here and elsewhere, we write
\begin{align}
 E_k ^{(\al, \be)} (t) &:=
\f { (2k+\al+\be +1)
\Ga(k+\al+\be +1)} {\Ga (k+\be +1)}P_k^{(\al,\be)}(t)\label{2-1}\\
&=c_{\a,\b} P_n^{(\a,\b)}(1)\|P_n^{(\a,\b)}\|_{2,\a,\b}^{-2}P_n^{(\a,\b)}(t),\notag
 \end{align}
where    $P_k^{(\al,\be)}$ is  the usual Jacobi polynomial of degree $k$ and   indices
$\al, \be$, as defined in \cite[Chapter IV]{Sz}, and \begin{equation}\label{Lebesgue-norm}
\|g\|_{p,\a,\b}:=\Bigl( \int_0^\pi |g(\cos\t)|^p (\sin\t/2)^{2\a+1}(\cos\t/2)^{2\b+1}\, dt\Bigr)^{\f1p}, \   0<p<\infty
\end{equation}
 for   $g:[-1,1]\to\RR$. Furthermore, throughout the paper, we always assume that $\a\ge\b\ge -\f12$.

 Using  (\ref{2-2}), one can extend the definition of $\proj_k$ to the whole space $L^1(\sph)$
so that there is a spherical harmonic expansion $f\backsimeq \sa(f):=\sum_{k=0}^\infty \proj_k(f)$
 associated to each $f\in L^1(\sph)$.
 The series $\sa(f)$ is called the Fourier-Laplace series of $f$ on $\sph$.  In the case
of  $d=2$, this is simply  the usual Fourier  series  of $2\pi$-periodic functions.
If $d\ge 3$, then given  any  $1\leq p\neq 2\leq \infty$,
there always exists a function $f\in L^p(\sph)$ such that the partial sum of the   Fourier-Laplace series
$\sa(f)$ does not converge in $L^p(\sph)$ (see \cite{BC}).  An important tool for the
investigation of summability of the series $\sa(f)$  is
to use the Ces\`aro means of $\sa(f)$, whose definition will be given below.

  {\it The Ces\`aro means of
  $\sigma(f)$ of order $\delta>-1$}  are  defined as usual by
\begin{equation}\label{cesaro:cho}\sigma_n^\delta(f):=\sum_{k=0}^n
\frac{A_{n-k}^\delta}{A_n^\delta} \proj_k(f),\   \   \ n=0,1,\cdots,
\end{equation}
 where  $A_k^\delta=\frac {\Gamma(k+\delta+1)} {\Gamma(k+1)\Gamma(\delta+1)}.$
 It is known that if $\d>\l:=\f{d-2}2$, and $f\in L^p(\sph)$ for $1\leq p<\infty$ or $f\in C(\sph)$ for $p=\infty$, then
  \begin{equation} \lim_{n\to\infty}\|\s_n^\d f-f\|_p =0.\end{equation}
  This result, in particular, implies that
  if $f, g\in L^1(\sph)$ satisfies  $\proj_j f=\proj_j g$ for all $j\ge 0$ then one must have  $f=g$.

   Another approach to spherical harmonic analysis is through the Laplace-Beltrami operator $\Delta_0$ on $\sph$ defined by
   \begin{equation}\label{2-6-eq} \Delta_0 f
:=\sum_{j=1}^d  \f{\p^2 F}{\p
x_j^2} \Bigl|_{\sph},
\   \  \text{ with  $F(y):=f\big(\frac{y}{|y|}\big)$}.\end{equation}
Indeed, each space   $\HH^d_k$ is  the space of eigenfunctions of  $\Delta_0$ corresponding to the
eigenvalue $-\l_k=-k(k+d-2)$; namely,
\begin{equation}\HH^d_k =\Bl\{ f\in C^2(\SS^{d-1}):\ \
\Delta_0 f =-\l_k f\Br\},\  \ k=0,1,\cdots.\end{equation}
  Therefore, spherical harmonic polynomial  expansions are simply the eigenvalue expansions of $\Delta_0$.

Given $r>0$,  we define the {\it fractional Laplace-Beltrami } operator  $(-\Delta_0)^{r}$  in a distributional  sense
  by
\begin{equation}\label{1:ch0} \proj_k\Bigl[(-\Delta_0)^r f \Bigr]= (k(k+d-2))^{r} \proj_k(f),\    \   \  k=0,1,\cdots.
 \end{equation}
 Clearly, if $r=1$, this definition coincides with the definition given in \eqref{2-6-eq}.

   Let $\eta$ be a nonnegative $C^\infty$-function on $\RR$
with the properties that $\eta(x)=1$ for $|x|\leq 1$ and
$\eta(x)=0$ for $|x|\ge 2$.   For each integer $n\ge 1$,
the generalized de la Vall\'ee Poussin operator is defined by
\begin{equation}\label{2-3} V_n f(x)=\sum_{k=0}^{2n} \eta(\f k n)\proj_k f(x)=\int_{\sph} f(y) K_n(x\cdot y)\, d\sa(y),\   \
x\in\sph,\end{equation}
where
\begin{equation}\label{1-9-TD}K_n(t)=C_d\sum_{k=0}^{2n} \eta(\f kn)
E_k^{(\f{d-3}2, \f{d-3}2)}(t),     \    \  t\in[-1,1].\end{equation}
We will keep the notations $\eta$, $V_n$ and $K_n$ for the rest of the paper.

It turns out that the kernel $K_n$  in \eqref{1-9-TD} is highly localized at the point $t=0$, as was shown in Lemma \ref{lem-1-1-sec1} below.
To be more precise, we define,
   for a smooth cutoff function  $\vi: [0,\infty)\to\CC$,
\begin{align}
&B_{N,\vi}^{(\a,\b)} (t):= \sum_{k=0}^{\infty}\vi( \f k N)   E_k^{(\a,\b)}
(t).\label{1-10-ch3}\end{align}
Then
the  following   pointwise estimates of the kernels
$B_{N,\vi}^{(\a,\b)}$ were known (
 \cite[Lemma 3.3]{BD} and \cite[Theorem 2.6]{IPX}):
\begin{lem}\label{lem-1-1-sec1}     Let   $\vi\in C^{3\ell-1}[0,\infty)$ be such that  $\text{supp}\, \vi\subset [0,2]$ and   $\vi^{(j)}(0)=0$ for $j=1,2,\cdots, 3\ell-2$.  Then for the kernel  function
$B_{N} \equiv B_{N,\vi}^{(\a,\b)} $   defined by \eqref{1-10-ch3} with $\a\ge \b\ge -1/2$,
\begin{equation}\label{1-2-ch3}
|B_{N} ^{(i)} (\cos\t) |\leq C_{\ell,i,\a} \|\vi^{(3\ell-1)}\|_\infty   N^{ 2\a+2i+2}
(1+N\t)^{-\ell},\    \    i=0,1,\cdots,\end{equation} where $\t\in [0,\pi]$, $N\in\NN$,
$ B_N^{(0)}(t)=
B_{N,\vi}^{(\a,\b)}(t)$ and $B_{N}^{(i)}(t)=
\(\f d{dt}\)^i \{
B_{N,\vi}^{(\a,\b)}(t)\}$ for $i\ge 1$.\end{lem}



We conclude this section with a technical theorem, which gives a sharp asymptotic estimate of the weighted $L^p$ norm of
the following kernel function:
\begin{align}
&G_{n,r}^{(\a,\b)} (t):= \sum_{k=0}^{\infty}\eta( \f k n) (k(k+\a+\b+1))^{\f r2} E_k^{(\a,\b)}
(t),\   \  r\ge 0.\label{1-1-ch3}   \end{align}
 For simplicity, we will   write $G_{n,r}$ for $G_{n,r}^{(\a,\b)}$, and
  $G_n$ for $G_{n,0}$, whenever $\a,\b$ are understood  and no confusion is possible from the context.
Recall that the norm $\|g\|_{p,\a,\b}$ is  defined by
 \eqref{Lebesgue-norm}.
\begin{thm}\label{prop-1-2-ch3} Let $G_{n,r}\equiv G_{n,r}^{(\a,\b)}$ be defined by \eqref{1-1-ch3}, and let $0<p< 1$ and $r>0$. Assume that
$r$ is not an even integer if $\a+\b+1>0$, and $r$ is not an integer if $\a+\b+1=0$. Then
\begin{equation}\label{3-7}
\f{\|G_{n,r} \|_{p, \a,\b} }{\|G_n\|_{p,\a,\b}}\sim \begin{cases} n^r,\    \   &\text{if $r>(2\a+2)(\f 1p-1)$},\\
n^{(2\a+2)(\f 1p-1)},\   \   &\text{if $r<(2\a+2)(\f 1p-1)$},\\
n^{(2\a+2)(\f 1p-1)}\log^{\f 1p} n,\    \   & \text{if $r=(2\a+2)(\f 1p-1)$.}\end{cases}
\end{equation}
\end{thm}
Theorem \ref{prop-1-2-ch3}
  will play a crucial role in the proof of Theorem \ref{sharp-bernstein}, whereas   its proof
 is quite technical. To avoid   interruption of our later discussion of various polynomial inequalities,
 we  postpone the proof of this theorem
  to the appendix section.

   More results on spherical harmonic expansions can be found in
  the book \cite{WL}.

 \section{Weighted polynomial inequalities}

In this section, we will review some known facts and results concerning  doubling weights, which will be useful in the remaining sections of the paper.
 \subsection{Doubling weights and properties}

Given a weight function $w$ on $\sph$,
we
write $ w(E):=\int_{E}w(x)\, d\sa(x)$ for a measurable $E\subset \sph$,
and
denote by $L_{p,w}\equiv L_{p,w}(\sph)$ the space of all real
  functions $f$ on $\sph$ with finite quasi-
norm
$$\|f\|_{p,w}:=\begin{cases}
\(\int_{\sph} |f(x)|^p w(x)\, d\sa(x)\)^{\f1p},&\   \    0<p<\infty,\\
\underset{x\in \sph}{\operatorname{esssup}} \, |f(x)|,&\   \
p=\infty.\end{cases}$$

A   weight function $w$  on $\sph$ is  said to satisfy {\it the doubling condition} if
there exists a constant $L>0$  such that
\begin{equation}\label{1-1} w(2B) \leq L
w(B)\   \  \  \text{for all spherical caps $B\subset \sph$}
,\end{equation}  where the least constant $L$  is called the
doubling constant of $w$, and is denoted by $L_w$. Following \cite{MT2},   we set, for  a given
doubling weight $w$ on $\sph$,
\begin{equation}\label{1--3} w_n(x) = n^{d-1} \int_{B(x,\f1n)} w(y)\, d\sa(y),\    \
n=1,2,\ldots, \   \  \text{and}\  \  w_0(x)=w_1(x). \end{equation}

 Define
\begin{equation}\label{2-3-ch3}
s_w':=\inf\Big\{ s\ge 0:\   \  \sup_{m\in\NN} \sup_B \f { w(2^m B)}{2^{ms}w(B)} <\infty\Big\},
\end{equation}    where the second  supremum on the right
 is  taken over all spherical caps $B\subset \sph$.
It can be shown that the number $s_w'$  exits and satisfies
$$
 \overline{\lim_{m\to\infty}} \f 1m{\log_ 2}\(\sup_B \f { w(2^m
B)}{w(B)}\)\leq s_w' \leq \f {\log L_w}{\log 2}.
$$

  We remark  that in many cases the infimum in \eqref{2-3-ch3} is attained at $s_w'$ and is computable. Taking the simple case
\begin{equation}\label{weight}w(x)=|x_1|^{\al_1}\cdots |x_d|^{\al_d},\   \    \min_{1\leq j\leq d}\a_j \ge 0\end{equation} for example, one has (see, e.g., \cite[(1.9)]{DW})
    $$s_w':= d-1 + \sum_{j=1}^d \al_j -\min_{1\leq j\leq d} \al_j.$$

From now on, we always assume that   $w$ is a doubling weight  on $\sph$   normalized by $\int_{\sph} w(y)\, d\s(y)=1$, we set  $s_w=s_w'$ if the infimum in \eqref{2-3-ch3} is achieved at $s_w'$, and otherwise, we  set $s_w$ to be a fixed constant satisfying $s_w'< s_w\leq {\log L_w}/{\log 2}$.
Unless otherwise stated, all  general constants $C$ below  depend only on $L_w$,  and the expression
$$\sup_{m\in\NN} \sup_B \f { w(2^m B)}{2^{ms_w}w(B)}$$
 whenever a doubling weight is involved.

 Using \eqref{1--3} and \eqref{2-3-ch3}, one can easily seen that
$$
w_n(x) \leq C 2^{s_w}( 1+n \rho(x, y ))^{s_w} w_n(y),\    \
\text{ $x,y\in\sph$,\   \  $n\ge 0$.}
$$

The following lemma collects some useful properties on doubling weights:

 \begin{lem}\label{lem-2-1-ch3}\cite[Section 2]{Dai1}
 Let $w$ be a doubling weight on $\sph$.

 (i) If $0<r<t$ and $x\in\sph$, then
 \begin{equation}\label{2-5-0-ch3} w( B( x,t)) \leq  C \Bigl( \f tr\Bigr)^{s_w} w (B (x,r)).\end{equation}

 (ii) For $x, y\in\sph$ and $n=0,1,\cdots$,
 \begin{equation}\label{2-4-ch3}w_n(x) \leq C ( 1+n \rho(x, y ))^{s_w} w_n(y).\end{equation}
 \end{lem}

 The following theorem was proved in  \cite[Corollary 3.4]{Dai1}.

 \begin{thm}\label{cor-3-4-TD}
For $f\in\Pi_n^d$ and $0<p<\infty$,
$$ C^{-1} \|f\|_{p,w_n} \leq \|f\|_{p,w} \leq C \|f\|_{p,w_n},$$
where $C>0$ depends only on $d$, $L_w$ and $p$ when $p$ is small.
\end{thm}

\subsection {A maximal function for spherical polynomials}

\begin{defn}\cite[(3.1)]{Dai1}
Given  $\xi >0$,   $f\in C(\sph)$ and $n\in\ZZ_+$, we define
\begin{equation}\label{3-2-ch3} f_{\xi,n}^\ast (x)=\max_{y\in\sph} |f(y)|( 1+n
 \rho(x,
 y))^{-\xi },\      \   x\in\sph.\    \
 \end{equation}
 \end{defn}

\begin{thm}\label{cor-3-2-ch3}\cite[Theorem 3.1]{Dai1}
If $ 0<p\leq \infty$, $ f\in\Pi^d_n$ and $\xi >\f{s_w}p$, then
 $$ \|f\|_{p,w} \leq \|f_{\xi,n}^\ast\|_{p,w} \leq C
 \|f\|_{p,w},$$where $C>0$ depends only on $d$, $L_w$ and $\xi$.
 \end{thm}

\subsection{Weighted cubature formulas and polynomial inequalities}

We start with the following definition.
\begin{defn} A  subset $\Lambda$ of \,$\sph$ is called
$\ve$-separated for some $\ve>0$  if
$ \rho( \o,
\o')\ge \ve$ for any two distinct points $\o,\o'\in\Lambda$.
A $\ve$-separated subset $\Lambda$ of
$\sph$ is called maximal if $\sph= \displaystyle \bigcup_{\o\in \Lambda} B (\o, \ve)$.
\end{defn}

From now on, let $\d_0$ be a sufficiently small constant depending only on $L_w$.

\begin{lem}\cite[Theorems 4.1]{Dai1}\cite[Lemma 3.4]{DW} \label{lem-2-1-TD}Given any   maximal $\f{\d} n$-separated  subset $\Ld$ of\, $
\sph$ with $\d\in (0,\d_0]$,  there exist positive numbers $\l_\o\sim w(B(\og, \f1N))$, $\og\in\Ld$, such that the following are true:
\begin{equation}\label{2-1-TD}
\int_{\sph} f(x) w(x)\, d\sa(x) =\sum_{\og\in\Ld} \ld_\og f(\og),\    \    \  \forall f\in \Pi_{4n}^d,
\end{equation} and
\begin{equation}\label{2-2-TD}
\|f\|_{p,w} \sim \begin{cases} \(\sum_{\og\in \Ld} \l_\og
|f(\og)|^p\)^{\f1p},&\  \  0<p<\infty,\\
\max_{\og\in \Ld} |f(\og)|,&\  \  p=\infty,
\end{cases}
\end{equation}
where  the constants of equivalence depend only on $L_w$, and
$p$ when $p$ is small.
\end{lem}

\begin{lem}\label{lem-2-6-TD}\cite[Lemma 2.3]{DW}
If  $0<p<q\leq \infty$, then
 \begin{equation}\label{Nil}\|f\|_{q,w}\leq C n^{(\f 1p-\f1q)s_w}\|f\|_{p,w},\    \    \forall f\in\Pi_n^d.\end{equation}
\end{lem}

\section{The Bernstein inequality with doubling weights}
In this section we study the sharp Bernstein inequality, that is, a sharp  growth on $n$ of the following expression:
$$\sup_{f\in \Pi_n^d, \|f\|_p\leq 1} \|(-\Delta_0)^{r/2} f\|_{L^p(\sph)}$$ or, more generally,
$$\sup_{f\in \Pi_n^d, \|f\|_{p,w}\leq 1} \|(-\Delta_0)^{r/2} f\|_{p,w}.$$
Theorem \ref{sharp-bernstein} in the introduction gives an answer to the first question, that is, in the unweighted case.
In the case of $d=2$, this result (for $0<p<1$)  is due to Belinskii and   Liflyand \cite{BL}, but their proof, especially for the lower estimates, does not work for the case of higher-dimensional spheres.

  For the proof of Theorem \ref{sharp-bernstein}, we first note that the lower estimates in \eqref{4-2-DT} of Theorem \ref{sharp-bernstein} follow directly from Theorem \ref{prop-1-2-ch3} with $\a=\b=\f{d-3}2$.  For the upper estimates in \eqref{4-2-DT}, we shall prove a more general weighted result for all doubling weights.

\begin{thm}\label{thm4-2-TD}If $d\ge 3$, $0<p<\infty$,  $r>0$, and $w$ is a doubling weight on $\sph$, then
\begin{equation}\label{4-3-DT}
\sup_{f\in \Pi_n^d, \|f\|_{p,w}\leq 1} \|(-\Delta_0)^{r/2} f\|_{p,w}\leq C \Phi(n,r,p),\end{equation}
where $$\Phi(n,r,p)= \begin{cases}
n^r, &\   \  \text{if $r>\d(p,w)$ or $r\in 2\NN$;}\\
n^r(\log n)^{\max\{\f 1p,1\}}, &\    \ \text{if $r=\da(p,w)$};\\
n^{\d(p,w)}, &\   \  \text{if $r<\da(p,w)$},
\end{cases}
$$
and
$$\da(p,w):=\begin{cases}
\f {s_w}p-(d-1),&\   \  \text{if $0<p\leq 1$};\\
\f {s_w-(d-1)}p,&\   \  \text{if $1< p<\infty$}.\end{cases}$$
\end{thm}

\begin{rem} (i) \   \  The proof of Theorem \ref{thm4-2-TD} below works equally well when $d=2$ and $r$ is not an integer, in which case (\ref{4-3-DT}) is  simply  the usual Bernstein inequality for the fractional derivatives of trigonometric polynomials, and to the best of our knowledge, our results for general doubling weights and non-integer $r$ are new.
Note also that in the case of $w=1$ (i.e., the unweighted case), $s_w=d-1$. Thus,  the upper estimate of  \eqref{4-2-DT} is a  direct consequence of Theorem \ref{thm4-2-TD}.

(ii)  \   \ Note that in the case when the power $r/2$ of the Laplace-Beltrami  operator  is an integer, then the weighted Bernstein
 inequality \eqref{4-1-TD-new} holds for the full range of $0<p<\infty$, whereas in the case of non-integer power, this is no longer true.

\end{rem}

The proof of Theorem \ref{thm4-2-TD} given below is different from that of \cite{BL}.\\

{\it Proof of Theorem \ref{thm4-2-TD}.}   Assume that $2^{m-1}\leq n<2^m$.  Define
$L_0 g= V_1 g$, and $L_j g=V_{2^{j}}f-V_{2^{j-1}}f$ for $j\ge 1$. Then for any $f\in \Pi_n^d$,
\begin{equation}\label{4-4-1-TD}(-\Delta_0)^{r/2}f=(-\Delta_0)^{r/2}(V_{2^m} f) =\sum_{j=0}^m (-\Delta_0)^{r/2}L_j f.\end{equation}
However, using \eqref{2-2} and \eqref{1:ch0}, it is easily seen that
\begin{equation}\label{4-4-TD} (-\Delta_0)^{r/2}L_j f(x)=\int_{\sph} f(y) L_{j,r}(x\cdot y)\, d\s(y),\end{equation}
where $L_{0,r} (t) =c_{d,r} E_1^{(\f{d-3}2, \f{d-3}2)}(t)$,
$$ L_{j,r}(t)= c\sum_{k=2^{j}}^{2^{j+2}} \psi(2^{-j}k)(k(k+d-2))^{r/2} E_k^{(\f{d-3}2, \f{d-3}2)}(t),\   \   j\ge 1,$$
and  $\psi(x)=\eta(x/2)-\eta(x)$.
Invoking  Lemma \ref{lem-1-1-sec1} with $\vi(x)=\psi(x) (x(x+2^{-j}(d-2)))^{r/2}$, we have
\begin{equation}\label{4-5-TD}
|L_{j,r}(\cos\t)|\leq c 2^{jr+d-1} (1+2^j\rho(x,y))^{-\ell},\    \  \forall \ell>1.
\end{equation}
Recalling the definition of  $w_k(x)$ in \eqref{1--3}, we obtain that  for $0<p\leq 1$,
\begin{align}
&|(-\Delta_0)^{r/2}L_j f(x)|^p w_{2^j}(x)\notag\\
&\leq c n^{(d-1)(1-p)}\int_{\sph} |f(y)|^p |L_{j,r}(x\cdot y)|^p\, d\s(y)w_{2^j}(x)\notag\\
&\leq c n^{(d-1)(1-p)}2^{jp(d-1+r)}\int_{\sph} |f(y)|^p (1+2^j \rho(x,y))^{-p\ell+s_w}w_{2^j}(y) \, d\s(y),\label{4-6-TD}\end{align}
where we used \eqref{4-4-TD} and  the unweighted  Nikolskii inequality (i.e.,  Lemma \ref{lem-2-6-TD} with $s_w=d-1$) in the first step, and used
\eqref{4-5-TD} and \eqref{2-4-ch3} in the second step.
Integrating this last inequality with respect to $x\in\sph$ gives
\begin{align}
&\|(-\Delta_0)^{r/2}L_j f\|^p_{p,w}\sim \|(-\Delta_0)^{r/2}L_j f\|^p_{p,w_{2^j}}\notag\\
&\leq C n^{-(d-1)p}2^{jp(d-1+r)} (n2^{-j})^{d-1} \int_{\sph}|f(y)|^p w_{2^j}(y)\, d\s(y)\notag\\
&\leq Cn^{-(d-1)p}2^{jp(d-1+r)} (n2^{-j})^{s_w} \int_{\sph}|f(y)|^p w_{n}(y)\, d\s(y)\notag\\
&\leq C n^{s_w-p(d-1)} 2^{j(p(d-1+r)-s_w)}\|f\|_{p,w}^p,\label{4-8-TD}
\end{align}
where the first step uses  Theorem \ref{cor-3-4-TD} and the fact that $(-\Delta_0)^{r/2}L_j f\in\Pi_{2^j}^d$. The second step uses  the inequality  \eqref{4-6-TD} with $\ell>(s_w+d-1)/p$, the third step uses \eqref{2-5-0-ch3}, and the last step follows from Theorem \ref{cor-3-4-TD} and the fact that $f\in\Pi_n^d$. Thus, combining \eqref{4-4-1-TD} with \eqref{4-8-TD}, we obtain
\begin{align*}
\|(-\Delta_0)^{r/2}(L_j f)\|_p^p \leq C \Bigl[n^{s_w-p(d-1)}\sum_{j=0}^m  2^{j(p(d-1+r)-s_w)}\Bigr]\|f\|_{p,w}^p,
\end{align*}
which, by straightforward calculation gives the desired upper bound.

The case of $p> 1$ can be treated similarly. Indeed,  instead of using Nikolskii's inequality, we use H\"older's inequality  to obtain
\begin{align*}&|(-\Delta_0)^{r/2} L_j f(x)|^p w_{2^j}(x)\leq C 2^{jr(p-1)} \int_{\sph} |f(y)|^p |L_{j,r} (x\cdot y)\, d\s(y)w_{2^j}(x)\\
&\leq C 2^{jrp} 2^{j(d-1)} \int_{\sph} |f(y)|^p ( 1+2^j\rho(x,y))^{-\ell+s_w} w_{2^j}(y)\, d\s(y)\\
&\leq C  2^{jrp}\Bigl( \f {2^j}{n}\Bigr)^{d-1-s_w} 2^{j(d-1)} \int_{\sph} |f(y)|^p ( 1+2^j\rho(x,y))^{-\ell+s_w} w_{n}(y)\, d\s(y).\end{align*}
We then integrate the last inequality with respect to $x\in\sph$ and deduce
$$\|(-\Delta_0)^{r/2} L_j f\|_{p,w}\leq C  2^{jr}\Bigl( \f {2^j}{n}\Bigr)^{(d-1-s_w)/p}\|f\|_{p,w},$$
which, in turn, implies
$$
\|(-\Delta_0)^{r/2}(L_j f)\|_p \leq  \sum_{j=0}^m \|(-\Delta_0)^{r/2} L_j f\|_{p,w}\leq C \Big(\sum_{j=0}^m 2^{jr}\Bigl( \f {2^j}{n}\Bigr)^{(d-1-s_w)/p}\Big)\,\|f\|_{p,w}.$$
The desired upper bounds for the case of $p>1$ then follow.
 \hb

\section{The Bernstein inequality with $A_p$ weights}

Given $1<p<\infty$, we say a weight function $w$ on $\sph$ belongs to $A_p$ if \begin{equation}\label{5-1-0}
\sup_{B} \f {w(B)}{|B|} \Bigl(\f 1{|B|} \int_{B} w(x)^{-\f {p'}p}\, d\s(x)\Bigr)^{p-1}\leq A_p(w)<\infty,
\end{equation}
where the supremum is taken over all the spherical caps $B$ of $\sph$.
A characterization of the Muckenhoupt $A_p$ condition  was recently obtained in \cite[Th. 2.4]{lerner}.

Similarly, a weight function $w$ belongs to $A_1$ if there exists a constant $A_1(w)>0$ such that
for all spherical caps $B\subset \sph$,
\begin{equation}\label{5-2-0}
\f 1{|B|}\int_B w(x)\, d\s(x)\leq A_1(w) \inf_{x\in B} w(x).
\end{equation}
It is well known that if $1<p<\infty$ and $w\in A_p$ then
\begin{equation}
\|Mf\|_{p,w}\leq C_p \|f\|_{p,w},
\end{equation}
where $Mf$ denotes the Hardy-Littlewood maximal function on $\sph$:
$$ Mf(x):=\sup_{0<r<\pi} \f 1{|B(x,r)|}\int_{B(x,r)} |f(y)|\, d\s(y).$$

Another useful fact on $A_p$ weights is the following: if $w\in A_p$ and $1\leq p<\infty$, then one can choose $s_w$ so that
\begin{equation}\label{5-3-0}
s_w \leq p(d-1);\end{equation}
see \cite[p. 196, (5)]{Stein95}.
Let us also mention that the $A_p$ classes have a self-improvement property (\cite[p. 202]{Stein95}), that is,
if $w\in A_p$ for some $1<p<\infty$, then  $w\in A_{p-\epsilon}$ for some $\epsilon > 0$.


   Using properties of the $A_p$-weights and  Theorem \ref{thm4-2-TD}, we can easily deduce the following   weighted Bernstein inequality for $A_p$ weights:

\begin{thm}\label{thm-6-2}
If $1\leq p<\infty$, $r>0$,
 $q:=\max\{ p, \f{d-1+pr}{d-1}\}$, and $w\in A_q$, then
 \begin{equation} \label{5-4-0}\|(-\Delta_0)^{r/2}f\|_{p,w}\leq C_{A_q(w)} n^r\|f\|_{p,w},\   \   \  \forall f\in\Pi_n^d.\end{equation}
 This, in particular, implies that if  $w\in A_p$, then \eqref{5-4-0} holds  for all $r>0$.
\end{thm}
\begin{proof} Firstly, we show \eqref{5-4-0} for the case of $r>(1-\f 1p)(d-1)$. In this case, $q=\f{d-1+pr}{d-1}$. Since  $w\in A_q$ implies that  $w\in A_{q-\va}$ for some small $\va>0$, using  \eqref{5-3-0}, we deduce  that
$s_w<q(d-1) =d-1+pr$, or equivalently,
$ r>\f {s_w-(d-1)}p.$    The desired inequality  \eqref{5-4-0} in this case  then follows from  Theorem \ref{thm4-2-TD}.

Next, we show \eqref{5-4-0} for $0<r\leq (1-\f 1p)(d-1)$, in which case $q=p$. If $p=1$ and $w\in A_1$ then using \eqref{5-3-0},  we have   $s_w=d-1$, and according to   Theorem \ref{thm4-2-TD},  \eqref{5-4-0} holds whenever  $r>s_w-(d-1)=0$.  Thus, it remains to show \eqref{5-4-0} for the case of   $w\in A_p$ and $1<p<\infty$. Observe that for all $f\in\Pi_n^d$,
\begin{equation}\label{5-6-0} (-\Delta_0)^{r/2} f(x)=(-\Delta_0)^{r/2} V_n f (x) :=\int_{\sph} f(y)K_{n,r}(x\cdot y)\, d\s(y),\end{equation}
where
\begin{equation}\label{5-7-0}K_{n,r}(\cos\t):=C_d \sum_{k=0}^{2n} \eta(\f kn) E_k^{(\f{d-3}2, \f{d-3}2)}(\cos\t).\end{equation}
Using Lemma \ref{lem-3-3-DT} with  $\a=\b=\f{d-3}2$,  we have
$$|K_{n,r} (\cos\t) |\leq c n^{d-1+r} ( 1+n\t)^{-(d-1+r)},\   \    0\leq \t\leq \pi.$$
Thus, a straightforward computation, using \eqref{5-6-0},  shows that for all $f\in\Pi_n^d$,
$$|(-\Delta_0)^{r/2}f(x)|\leq c n^r Mf(x),  \    \  x\in\sph. $$
Since $w\in A_p$ and $1<p<\infty$, this implies that
$$ \|(-\Delta_0)^{r/2}f\|_{p,w}\leq c n^r \|Mf\|_{p,w}\leq c n^r\|f\|_{p,w},$$
which is the desired Bernstein inequality.
\end{proof}

\section{Weighted characterization of the Bernstein inequality}

\begin{defn} Given $1< p<\infty$, and $\tau\ge 0$, we say a weight function $w$ on $\sph$ belongs to the class $\mathbb{A}_{p,\tau}$ if   for any $r>\tau$,
\begin{equation}\label{6-1-0}
\sup_{B}\sup_{n\in\NN} \f {w_n(B)}{|B|} \Bigl( \f 1 {|B|} \int_B w_n(y)^{-
\f 1{p-1}
} \, d\s(y) \Bigr)^{p-1} ( 1+ n |B|^{\f 1{d-1}})^{-rp}
= \mathbb{A}_{p,\tau}(w)<\infty,
\end{equation}
where the first supremum is taken over all spherical caps of \,$\sph$. We say $w\in \mathbb{A}_{1,\tau}$ if there exists a constant $C>0$ such that for all spherical caps $B\subset \sph$, and all $r>\tau$,
\begin{equation}\label{6-2-0}
\f {w_n(B)}{|B|}\leq C (1+n|B|^{\f 1{d-1}})^r \inf_{x\in B} w_n(x).\end{equation}
The smallest value of $C$ in (\ref{6-2-0}) is called the $\mathbb{A}_{1,\tau}(w)$ constant. 
\end{defn}
The following lemma collects some useful properties on weights from the class $\mathbb{A}_{p,\tau}$.
\begin{lem}\label{lem-6-2}\begin{enumerate}[\rm (i)]
\item  If $0<\tau_1\le \tau_2$  and $1\leq p<\infty$, then $\mathbb{A}_{p,\tau_1}\subset \mathbb{A}_{p,\tau_2}$.
    \item If $1\le p\le q<\infty$ and $\tau>0$, then
$\mathbb{A}_{p,\tau}\subset \mathbb{A}_{q,\tau}$.
\item If $w$ is a doubling weight on $\sph$, then  $w\in \mathbb{A}_{p,\tau}$ with $\tau:= \f {s_w-(d-1)}{p}$.
    \item For any $1\leq p<\infty$, we have
    $$ A_p \subset \bigcup_{\tau>0} \mathbb{A}_{p,\tau}.$$ 
\item $w\in \mathbb{A}_{p,\tau}$ if and only if for any  $f\in L(\sph)$,  any spherical cap $B:=B(x,\t)\subset \sph$, and any $r>\tau$,
\begin{equation}\label{6-0-0} |f_B|^p\leq C (1+n\t)^{rp} \f 1{w_n(B)}\int_B |f(y)|^p w_n(y)\, d\s(y),\end{equation}
where $f_B:=\f 1{|B|} \int_B f(y)\, d\s(y)$, and the constant $C$ is independent of $B$, $f$ and $n$. \end{enumerate}
\end{lem}
\begin{proof} Assertion  (i) is obvious from the definition of the $\mathbb{A}_{p,\tau}$ class.
 Assertion (ii) follows by H\"older's inequality and the fact that the term on the left hand side of \eqref{6-1-0} is a decreasing function of $p$.

 To prove Assertion  (iii) for  the case of $p>1$, it suffices to show that for  $B=B(x,\t)\subset \sph$, and $\tau:= \f {s_w-(d-1)}{p}$,
\begin{equation}\label{6-3-0}
\f {w_n(B)}{|B|} \Bigl( \f 1 {|B|} \int_B w_n(y)^{-\f 1{p-1}}\, d\s(y) \Bigr)^{p-1} \leq C (1+n\t)^{\tau p}.
\end{equation}
 \eqref{6-3-0} holds trivially if $\t\leq \f 1n$ since $w_n(y) \sim w_n(z)$ whenever $\rho(y,z) \leq n^{-1}$. Now assume that $\f 1\t \sim m$ for some  positive integer $m\leq n$. Then  $w_m(y)\sim w_m(x)$ whenever $y\in B$. Since $m\leq n$, it is easily  seen that $$w_n(B)\sim w(B)\sim |B| w_m(x),$$
 and  using Lemma \ref{lem-2-1-ch3}, we deduce
$$ \f {w_n(y)}{w_m(y)} \ge c \Bigl( \f nm\Bigr)^{d-1-s_w},\   \ y\in\sph.$$
Thus,
\begin{align*}
\f {w_n(B)}{|B|} &\Bigl( \f 1 {|B|} \int_B w_n(y)^{-\f 1{p-1}}\, d\s(y) \Bigr)^{p-1}\\
&\leq c w_m(x) \Bigl( \f 1 {|B|} \int_B w_m(y)^{-\f 1{p-1}}\, d\s(y) \Bigr)^{p-1} \Bigl( \f nm\Bigr)^{s_w-(d-1)}\\
&\leq c \Bigl( \f nm\Bigr)^{s_w-(d-1)}\sim (n\t)^{s_w-d+1}\leq c (n\t)^{rp},\end{align*}
provided that $r\ge  (s_w-d+1)/p$. This proves Assertion (iii) for the case $p>1$. Assertion (iii) for the case $p=1$ can be treated similarly.

Assertion (iv) follows directly from Theorem \ref{thm-6-2} and Theorem \ref{thm-7-4} below.

 Finally, we show assertion (v). We first prove the necessity. Again we just deal with the case of $p>1$ for the sake of simplicity. Using  H\"older's inequality and the $\mathbb{A}_{p,\tau}$-condition, we have, for $r>\tau$,
\begin{align*}
|f_B|^p &\leq \Bigl(\f 1{|B|}\int_B |f(y)| w_n(y)^{\f 1p} w_n(y)^{-\f1p}\, d\s(y) \Bigr)^p \\
&\leq \Bigl(\f 1{|B|}\int_B |f(y)|^p w_n(y)  d\s(y) \Bigr) \Bigl( \f 1{|B|}\int_B w_n(y)^{-\f {p'}p}\, d\s(y)\Bigr)^{p-1}\\
&\leq c
( 1+ n |B|^{\f 1{d-1}})^{rp}
\f 1{w_n(B)}\int_B |f(y)|^p w_n(y)  d\s(y).\end{align*}
This proves that the $\mathbb{A}_{p,\tau}$-condition \eqref{6-1-0}  implies the condition \eqref{6-0-0}. Finally, the sufficiency part of Assertion (v)  follows directly by setting $f(x) =w_n(x)^{-\f 1{p-1}}$.
\end{proof}
The next result was proved in \cite[Lemma 2.5]{Dai2}.
\begin{lem} \label{lem-7-3}If $1\leq p<\infty$, and $w$ is a doubling weight, then
$$\|V_n f\|_{p,w_n} \leq c \|f\|_{p,w_n},\   \   \forall f\in L_{p},\   \   \forall n\in\NN.$$
\end{lem}

Before stating the main result in this section, we
  recall that,
 if the power $r/2$ is a positive integer, then for all  doubling weights $w$, the weighted Bernstein inequality \eqref{4-1-TD-new}
holds for the full range of  $0<p< \infty$, while this is no longer true when the power $r/2$ is non-integer. Indeed, for the latter case, we have the following main theorem,
which characterizes  those  weights $w$ for which
 the weighted Bernstein inequality
 \eqref{4-1-TD-new} holds.

\begin{thm}\label{thm-7-4} Assume that  $1\leq p<\infty$,  $w$ is  a doubling weight on $\sph$, and  $\tau\ge 0$.   Then the weighted Bernstein inequality  \eqref{4-1-TD-new},
with the constant $C$ independent of $n$ and $f$, holds for all $r>\tau$ if and only if $w\in\mathbb{A}_{p,\tau}$.
\end{thm}
\begin{rem} Note that  Theorem \ref{thm-7-4} is new even in the case of trigonometric polynomials (i.e., $d=2$).
 Next, we would like to
  remark that the sufficiency part of this theorem implies
   Theorem \ref{thm4-2-TD} for $1\le p<\infty$. Indeed,
if $w$ is a doubling weight, then by  Lemma \ref{lem-6-2} (iii),
 $w\in \mathbb{A}_{p,\tau}$ with $\tau:= \f {s_w-(d-1)}{p}$
 and by Theorem \ref{thm-7-4},
  the Bernstein inequality  \eqref{5-4-0} holds.
\end{rem}

\begin{proof}Firstly, we show that if the weighted Bernstein inequality \eqref{5-4-0} holds for some positive $r\notin 2\NN$, then $w\in \mathbb{A}_{p,r}$.
Let $K\ge 5$ be a sufficiently large constant and $\va\in (0,1)$ a sufficiently small constant, both depending only on the dimension $d$. Let $B=B(x,\t)$ with $x\in \sph$ and $\f K n <\t\leq \f{\va}K$. Let $x_2\in\sph$ be such that $ 4\t\leq \rho(x, x_2) \leq K\t\leq \va$. Let $B_2:=B(x_2, \t)$. Then for a nonnegative function  $f$  supported in $B:=B(x,\t)$, and  an arbitrary $z\in B_2$, we have
\begin{align*}
|(-\Delta_0)^{r/2} V_n f(z)|&=\Bigl|\int_{B} f(y) K_{n,r}(z\cdot y)\, d\s(y) \Bigr|\sim \int_{B} f(y) \rho(z,y)^{-(d-1+r)}\, d\s(y)\\
&\sim \t^{-r} \f 1{|B|} \int_B f(y)\, d\s(y)\equiv \t^{-r} f_B,\end{align*}
where we used Lemma \ref{lem-3-4-TD} in the second step.
On the other hand, using the weighted Bernstein inequality \eqref{5-4-0},  we obtain
\begin{align*}
 |\t^{-r} f_B |^p w_n(B_2) &\leq c \|(-\Delta_0)^{r/2} V_n f\|_{p,w_n}^p \leq c n^{rp} \|V_n f\|_{p,w_n}^p\\
& \leq c n^{rp} \|f\|_{p,w_n}^p.
\end{align*}
Thus, for any nonnegative function $f$ supported in $B$,
\begin{equation}\label{6-5-0}
|f_B|^p w_n (B_2) \leq c (n\t)^{rp} \int_{B} |f(y)|^p w_n(y)\, d\s(y).\end{equation}
Since $w$ is a doubling weight, $w_n$ satisfies the doubling condition as well with $L_{w_n}\leq c L_w$. Since $B\subset B(x_2, 2K\t) =
2K B_2$, it follows that
$$ w_n (B_2) \ge c w_n ( 2K B_2) \ge c w_n (B).$$
This combined with \eqref{6-5-0} yields
$$ |f_B|^p \leq c (n\t)^{rp} \f 1{ w_n (B)} \int_{B} |f(y)|^p w_n(y)\, d\s(y).$$
Letting $f(y) =w_n(y)^{-\f {p'} p} \chi_B (y),$
 we conclude that
 \begin{equation}\label{6-6-0} \f {w_n(B)} {|B|} \Bigl( \f 1{|B|} \int_B w_n(y)^{-\f {p'}p} \, dy \Bigr)^{p-1}\leq c (1+n\t)^{rp},\end{equation}
 whenever $B=B(x,\t)$ with  $n^{-1} K\leq \t\leq \va/K$.
On the other hand, since $w_n(x) \sim w_n(y)$ whenever $\rho(x,y) \leq c n^{-1}$, \eqref{6-6-0} holds trivially if $n\t\leq K$.

 Next, we  show \eqref{6-6-0} for the case of $B=B(x,\t)$ and   $\va/K \leq \t\leq \pi$.  We first observe that
 $$
 \f {w_n(B)}{|B|} \Bigl( \f 1{|B|} \int_B w_n(y)^{-\f {p'}p}\, d\s(y) \Bigr)^{p-1}  \leq c_\va   \Bigl(  \int_B w_n(y)^{-\f 1{p-1}}\, d\s(y) \Bigr)^{p-1}.$$
 Since the ball $B=B(x,\t)$ can be covered by a number of $\leq C_d \va^{-d+1}$  spherical caps  of radius $\leq \va/K$, it follows that
  $$  \int_B w_n(y)^{-\f 1{p-1}}\, d\s(y) \leq C_\va \sup_{\text{rad} (B')=\va/K}  \int_{B'} w_n(y)^{-\f 1{p-1}}\, d\s(y).$$
  On the other hand, using the doubling   condition, it is easily seen that if $B'$ is a spherical cap with radius $\va/K$, then
 $w_n (B') \ge c_{\va} w_n (\sph) \ge c_\va'>0$.  Therefore we get
 \begin{align*}
 \Bigl(  \int_B w_n(y)^{-\f 1{p-1}}\, d\s(y) \Bigr)^{p-1} &\leq C_\va \sup_{\text{rad} (B')=\va/K} \Bigl( \int_{B'} w_n(y)^{-\f 1{p-1}}\, d\s(y)\Bigr)^{p-1}\\
 &\leq  C_\va \sup_{\text{rad} (B')=\va/K} \Bigl( \f 1{|B'|} \int_{B'} w_n(y)^{-\f 1{p-1}}\, d\s(y) \Bigr)^{p-1} \f {w_n(B')}{|B'|}
 \\
 &\leq c \,n^{rp}\leq c\, (n\t)^{rp},
 \end{align*}
where the third step uses \eqref{6-6-0} for the already proven case $\t=\va/K$. This completes the proof of necessity.

To show the sufficiency, we assume that $w\in \mathbb{A}_{p,\tau}$ and $r>\tau$.   Then for $f\in\Pi_n^d$,
\begin{align*}
(-\Delta_0)^{r/2} f(x) =(-\Delta_0)^{r/2} V_n f(x)=\int_{\sph} f(y) K_{n,r}(x\cdot y)\, d\s(y).
\end{align*}
Using Lemma \ref{lem-3-3-DT} and integration by parts, we have
\begin{align*}
|(-\Delta_0)^{r/2} f(x)|&\leq c n^{d-1+r} \int_{\sph} |f(y)|(1+n\rho(x,y))^{-(d-1+r)}\, d\s(y)\\
 &\leq  c \|f\|_1 + c n^{1+r} \int_0^\pi (1+n\t)^{-(1+r)}
|f_{B(x,\t)}|\, d\t,\\
&\leq  c n^r \|f\|_{p,w_n} +cJ(x),
\end{align*}
where  $J(x):=  n^{1+r} \int_0^\pi (1+n\t)^{-(1+r)}
|f_{B(x,\t)}|\, d\t$.  To estimate $J(x)$, we let Let $r_1\in (\tau, r)$, and  choose $\a, \b$ so that $\a+\b=1+r$, $\a>r_1+\f 1p$ and $\b>\f 1{p'}:=1-\f1p$. Then using H\"older's inequality, we obtain
\begin{align*}
J(x)^p&\leq c n^{(1+r)p}\Bigl( \int_0^\pi (1+n\t)^{-\a p}
|f_{B(x,\t)}|^p\, d\t\Bigr)\Bigl( \int_0^\pi (1+n\t)^{-\b p'}\, d\s(y)\Bigr)^{p-1}\\
&\leq c n^{rp+1}  \int_0^\pi (1+n\t)^{-\a p+r_1p} \f 1{ w_n (B(x,\t))}\int_{B(x,\t)} |f(y)|^p w_n (y)\, d\s(y)
 d\t ,\end{align*}
where we used Assertion (ii) of  Lemma \ref{lem-6-2} in the second step.
For $\t\in (0,\pi)$, let $\Ld_\t$ be a maximal $\t$-separated subset of $\sph$. Then
\begin{align*}
&\int_{\sph}  \Bigl[\f 1{ w_n (B(x,\t))}\int_{B(x,\t)} |f(y)|^p w_n (y)\, d\s(y)\Bigr] w_n(x)\, d\s(x)\\
&\leq \sum_{\o\in\Ld_\t} \int_{B(\o, \t)}  \Bigl[\f 1{ w_n (B(x,\t))}\int_{B(x,\t)} |f(y)|^p w_n (y)\, d\s(y)\Bigr] w_n(x) \, d\s(x)\\
&\leq c \sum_{\o\in\Ld_\t} \int_{B(\o, \t)}  \Bigl[\f 1{ w_n (B(\o,\t))}\int_{B(\o,3\t)} |f(y)|^p w_n (y)\, d\s(y)\Bigr]w_n(x)\, d\s(x)\\
&\leq c \sum_{\o\in\Ld_\t}   \int_{B(\o,3\t)} |f(y)|^p w_n (y)\, d\s(y)\leq c \|f\|_{p,w_n}^p,
\end{align*}
where the third step uses the doubling condition of $w_n$.
Thus,
\begin{align*}\|J\|_{p,w_n}^p\leq c n^{rp+1}  \int_0^\pi (1+n\t)^{-\a p+r_1p}\, d\t \, \|f\|_{p,w_n}^p\leq  c n^{rp} \|f\|_{p,w_n}^p.\end{align*}
This completes the proof of the sufficiency.
\end{proof}

Theorem \ref{thm-7-4} implies the following  interesting corollary  on the weighted   Bernstein inequality with respect to doubling weights.

\begin{cor} Given  a   doubling weight $w$ on $\sph$ with $d\ge 3$, if the weighted Bernstein inequality \eqref{5-4-0} holds for  some $p=p_1\in [1,\infty)$ and some positive number  $r=r_1$ which is not an even integer, then automatically, it holds for all $p_1\leq p<\infty$ and $r\ge r_1$.
\end{cor}

\begin{proof} Firstly, note that from the proof of Theorem \ref{thm-7-4}, if \eqref{5-4-0} holds for   $p=p_1\in [1,\infty)$ and $r=r_1\notin 2\NN$, then $w\in\mathbb{A}_{p_1, r_1}$. Since $\mathbb{A}_{p_1, r_1}\subset \mathbb{A}_{p, r}$ for all  $p\ge p_1$ and $r\ge r_1$, Theorem \ref{thm-7-4} implies that \eqref{5-4-0}  holds for all   $p_1\leq p <\infty$ and $r>r_1$. Thus, it remains to  show \eqref{5-4-0} for the case of $r=r_1$ and $p_1<p<\infty$. To see this, we first note that  for all $F\in L_{p_1}$,
\begin{align*}
&\|(-\Delta_0)^{{r_1}/2}V_n F\|_{p_1,w}\leq C  n^{r_1}\|V_n F\|_{p_1,w}\leq C n^{r_1}\|V_n F\|_{p,w_n}\leq C n^{r_1}\|F\|_{p,w_n},\end{align*}
where we used \eqref{5-4-0} with $r=r_1$ and $p=p_1$ in the first step, Theorem \ref{cor-3-4-TD} in the second step, and  Lemma \ref{lem-7-3} in the last step.
On the other hand,  using the  unweighted Bernstein inequality, and the boundedness of the operator $V_n$ on $L_\infty$,
$$\|(-\Delta_0)^{{r_1}/2}
V_n F\|_{\infty}\leq C n^{r_1} \|V_n F\|_\infty\leq C n^{r_1} \|F\|_\infty,\  \  \forall F\in L_\infty.$$
Thus, applying the Riesz-Thorin interpolation theorem, we deduce that
$$ \|(-\Delta_0)^{{r_1}/2}V_n F\|_{p,w}\leq C n^{r_1}\|F\|_{p,w_n},\   \    \forall F\in L_{p},\   \ p_1\leq p\leq \infty.$$
To complete the proof, we just note that $V_n f=f$ and
$\|f\|_{p,w}\sim \|f\|_{p,w_n}$ for all $f\in\Pi^d_n.$
\end{proof}

We conclude this section with the following example.

\begin{exam} Let  $w(x)=\prod_{j=1}^d |x_j|^{\a_j}$ and $1\leq p<\infty$.  From the proof of Proposition 6.1 in \cite{DW}, it is easy to verify that  if $\min_{1\leq j\leq d} \a_j > p-1$, then $w\in \mathbb{A}_{p,\tau_{w,p}}$  but $w\notin \mathbb{A}_{p, \xi}$ for any $\xi<\tau_{w,p}$, where $$
\tau_{w,p}:=\f {s_w}p -(d-1) =\f1p\Bigl(\sum_{j=1}^d \a_j -\min_{1\leq j\leq d} \a_j\Bigr) -(1-\f1p)(d-1).
$$
   Thus, in this case, the weighted Bernstein inequality  \eqref{5-4-0} holds for  $r>\tau_{w,p}$, and fails for $0<r<\tau_{w,p}$.
\end{exam}

\section{Approximation in $L^p$-spaces with $0<p<1$}

Recall that the generalized de la Vall\'ee Poussin mean  $V_n f$ is defined by \eqref{2-3} for all $f\in L_p$ with $1\leq p\leq \infty$. It can be easily seen from the definition that
$V_n g=g$ for $g\in \Pi_n^d$, and for all $f\in L_p$ with $1\leq p\leq \infty$, \begin{equation}\label{2-11-eq}
E_{2n}(f)_p \leq \|f-V_n f\|_p\leq C_d E_n(f)_p.
\end{equation}
This last fact, however, cannot be true for $0<p<1$, in which case, $V_n f$ is not  even  defined for all $f\in L_p$. In this section, we shall prove that given a function $f\in L_p$ with $0<p<1$, there always exists a Fourier-Laplace series on $\sph$ whose generalized de la Vall\'ee Poussin mean converges to $f$ in $L_p$-norm, and an estimate weaker than \eqref{2-11-eq} remains true. The idea of using generalized de la Vall\'ee Poussin means of Fourier series to approximate  functions in $L_p$ with $0<p<1$ goes back to Oswald
\cite{osw}.

Given a Fourier-Laplace series
\begin{equation}\label{5-1-TD}\sa\sim \sum_{k=0}^\infty Y_k(x),\   \  Y_k\in\HH_k^d,\end{equation}
we define
$ V_n \sa :=\sum_{k=0}^{2n} \eta(\f kn) Y_k(x),$ and
and $S_n \sa:=\sum_{k=0}^n Y_k(x).$

Our main result in this section is the following.

\begin{thm}\label{thm-5-1-TD} If $0<p<1$ and $f\in L^p(\sph)$, then
there exists a Fourier-Laplace series $\sa$ of the form \eqref{5-1-TD} such that
\begin{equation}\label{5-2-TD}
\|f-V_n\sa\|_p \leq C_p n^{-(d-1)(\f 1p-1)} \Bigl( \sum_{k=1}^n k^{d-2-(d-1)p} E_k (f)_p^p\Bigr)^{\f1p}.\end{equation}
If, in addition,
\begin{equation}\label{5-3-TD}
\sum_{k=1}^\infty k^{d-2-(d-1)p}E_k(f)_p^p<\infty,\end{equation}
then $f\in L_1$,   and one has the following stronger estimate:
$$ \|f-V_n f\|_p\leq C n^{-(d-1)(\f 1p-1)} \Bigl( \sum_{k={n+1}}^\infty k^{(d-2)-(d-1)p}E_k(f)_p^p\Bigr)^{\f1p}.$$
\end{thm}

\begin{rem}
It is worth mentioning that the term on the right-hand side of (\ref{5-2-TD}) tends to $0$ as $n\to \infty$ and therefore
(\ref{5-2-TD}) can be considered as a generalization of Oswald's result \cite{osw} on $\sph$.
\end{rem}

In the case of periodic functions, Theorem \ref{thm-5-1-TD} is due to Belinskii and   Liflyand \cite{BL}.

The proof of Theorem \ref{thm-5-1-TD} relies on several lemmas.
\begin{lem}\label{lem-5-2-TD}
Assume that    $f\in \Pi_{6n}^d$, and  $G_n:[-1,1]\to\RR$ is an algebraic polynomial of degree at most $n$. If $0<p<1$,  then
\begin{equation}\label{5-4-TD}
\Bigl( \int_{\sph} \Bigl|\int_{\sph} f(y) G_n(x\cdot y)\, d\s(y)\Bigr|^p\, d\s(x)\Br)^{\f1p}\leq C n^{(\f 1p-1)(d-1)}\|f\|_p\|G_n\|_{p,\a,\b},\end{equation}
where $\a=\b=\f{d-3}2$.
\end{lem}
\begin{proof} The desired inequality \eqref{5-4-TD} follows directly from  Lemma \ref{lem-2-6-TD} applied to  $w=1$, $q=1$ and $0<p<1$:
\begin{align*}
 \int_{\sph} \Bigl|\int_{\sph} f(y) G_n(x\cdot y)\, &d\s(y)\Bigr|^p\, d\s(x)\\
 &\leq C n^{(d-1)(1-p)} \int_{\sph} \int_{\sph} |f(y) G_n(x\cdot y)|^p\, d\s(y)\, d\s(x)\\
 &=C n^{(d-1)(1-p)}\|f\|_p^p\|G_n\|_{p,\a,\b}^p.
\end{align*}
\end{proof}

\begin{lem}\label{lem-5-3-TD}
If $0<p<1$,  and $f\in \Pi_{6n}^d$, then
$$\|V_n f\|_p\leq C_p\,\|f\|_p.$$\end{lem}
\begin{proof}By \eqref{2-3}, we have
$$ V_n f(x)=\int_{\sph}f(y)K_n(x\cdot y)\, d\s(y),$$
where $K_n=G_n^{(\f{d-3}2,\f{d-3}2)}$ is given by \eqref{1-9-TD}.
Thus, using Lemma \ref{lem-5-2-TD} with $\a=\b=\f{d-3}2$, we deduce
$$\|V_nf\|_p\leq C  n^{(\f 1p-1)(d-1)}\|f\|_p\|G_n^{(\a,\b)}\|_{p,\a,\b}\leq C \|f\|_p,$$
where the last step uses \eqref{3-6}.
\end{proof}

The following lemma plays a crucial role in the proof of Theorem \ref{thm-5-1-TD}.
\begin{lem}\label{lem-5-4-TD}Assume that $f\in \Pi_k^d$, and  $0<p<1$. Then there exists a Fourier-Laplace series of the form \eqref{5-1-TD} such that $S_k \sa=f$, and for all $n\ge k$,
\begin{equation}
\|V_n\sa\|_p \leq C \Bigl( \f kn\Bigr)^{(d-1)(\f 1p-1)}\|f\|_p.\end{equation}
\end{lem}
\begin{proof}  Let $\Ld_k$ be a maximal $\f {\d}k$-separated subset of $\sph$, with  $\d\in (0,1)$ being a small constant depending only on $d$.
We denote by $N_k$ the number of points in the set $\Ld_k$. Then $N_k\sim k^{d-1}$, and by Lemma \ref{lem-2-1-TD}, there exists a positive cubature formula of degree $k$ on $\sph$,
\begin{equation}\label{5-6--TD}
\int_{\sph} P(y)\, d\s(y) =\sum_{\o\in\Ld_k} \l_\o P(\o),\   \   \forall P\in\Pi_k^d,
\end{equation}
such that $\l_\o\sim N_k^{-1}$ for all $\o\in\Ld_k$.  Define
\begin{equation}\label{5-6-TD}\s(x):=\sum_{j=0}^\infty  \sum_{\o\in\Ld_k}\l_\o E_j^{(\f{d-3}2, \f{d-3}2)}(x\cdot \og) f(x).\end{equation}
Clearly, by the cubature formula \eqref{5-6--TD},
\begin{align*}
S_k\sa&=\sum_{j=0}^k  \Bigl[\sum_{\o\in\Ld_k}\l_\o E_j^{(\f{d-3}2, \f{d-3}2)}(x\cdot \og)\Bigr] f(x)\\
&=f(x)\sum_{j=0}^k \int_{\sph}E_j^{(\f{d-3}2, \f{d-3}2)}(x\cdot y)\, d\s(y) =f(x).\end{align*}

Since for each $\o\in\Ld_k$, $E_j^{(\f{d-3}2, \f{d-3}2)}(x\cdot \og)$, as a function of $x\in\sph$, is a spherical harmonic of degree $j$, it follows that
$$  E_j^{(\f{d-3}2, \f{d-3}2)}(x\cdot \og) f(x)\in \sum_{\ell=|k-j|}^{k+j}\HH_\ell^d.$$
We can rewrite \eqref{5-6-TD} in the form $\sum_{m=0}^\infty Y_m(x)$, with
$$ Y_m(x) =\sum_{j=0}^{m+k} \f 1{N_k} \sum_{\o\in\Ld_k}\proj_m\Bigl[fE_j^{(\f{d-3}2, \f{d-3}2)}(\la\cdot,  \og\ra) \Bigr](x)\in \HH_m^d.$$
This also implies that  $V_n^{(x)}[ E_j^{(\f{d-3}2, \f{d-3}2)}(x\cdot \og) f(x)]=0$ whenever $j\ge 2n+k$, where we use the notation $V_n^{(x)}$ to mean that the operator $V_n$ acts on the variable $x$.  Therefore, setting   $P_j=E_j^{(\f{d-3}2, \f{d-3}2)}$, we obtain
\begin{align*}
V_n\sa&=\f 1{N_k} \sum_{\o\in\Ld_k} V_n^{(x)} \Bigl[ \sum_{j=0}^{2n+k} P_j(x\cdot \og) f(x)\Bigr]=\f 1{N_k} \sum_{\o\in\Ld_k}V_n^{(x)} \Bigl[ \sum_{j=0}^{6n} \eta(\f j{3n})P_j(x\cdot \og) f(x)\Bigr]\\
&=\f 1{N_k} \sum_{\o\in \Ld_k} V_n^{(x)} \Bigl[ K_{3n} (x\cdot \og) f(x)\Bigr],
\end{align*}
where the function $K_{3n}$ is defined by \eqref{1-9-TD}.  Letting $\xi>\f{d-1}p$,  we have
\begin{align*}
\|V_n \sa\|_p^p &\leq C N_k^{-p}\sum_{\o\in\Ld_k}\int_{\sph} |K_{3n} (x\cdot \og) f(x)|^p\, d\s(x)\\
&\leq
C  \Bigl( N_k^{-p}\sum_{\o\in\Ld_k} |f_{k,\xi}^\ast (\og)|^p \Bigr)\sup_{y\in \sph} \int_{\sph} |K_{3n}(x\cdot y)|^p (1+k \rho(x,y))^{\xi p}\, d\s(x) \\
&\leq C N_k ^{1-p} \|f_{k,\xi}^\ast\|_p^p \sup_{y\in \sph} \int_{\sph} |K_{3n}(x\cdot y)|^p (1+n \rho(x,y))^{\xi p}\, d\s(x)
\\
&\leq C k^{(d-1)(1-p)} n^{(d-1)(p-1)} \|f\|_p^p= C \Bigl( \f kn\Bigr)^{(d-1)(1-p)}\|f\|_p^p,
\end{align*}
where we used Lemma \ref{lem-5-3-TD} in the first step, the maximal function defined by \eqref{3-2-ch3} in the second step, and Theorem \ref{cor-3-2-ch3} and Lemma \ref{lem-1-1-sec1}  in the last step. This completes the proof.
\end{proof}

\begin{lem}\label{lem-5-5-TD}If $0<p<1$ and $f\in L^p(\sph)$ satisfies \eqref{5-3-TD}, then $f\in L(\sph)$ and
$$\int_{\sph} |f(x)|\, d\s(x) \leq C \Bigl( \sum_{k=1}^\infty k^{d-2-(d-1)p} E_k (f)_p^p\Bigr)^{\f1p} +C\|f\|_p.$$
\end{lem}
\begin{proof}
Let $f_j\in\Pi_{2^j}^d$ be such that $E_{2^j} (f)_p :=\|f-f_j\|_p$ for $j\ge 0$. Then by Fatou's lemma, we have
\begin{align*}
\|f\|_1 &\leq \|f_0\|_1+\sum_{j=1}^\infty \|f_j-f_{j-1}\|_1\leq C\|f\|_p + C \sum_{j=1}^\infty 2^{j (d-1) (\f 1p-1)}\|f_j-f_{j-1}\|_p\\
&\leq C\|f\|_p + C \sum_{j=1}^\infty 2^{j (d-1) (\f 1p-1)}E_{2^{j-1}}(f)_p\leq  C\|f\|_p + C \Bigl(\sum_{j=1}^\infty 2^{j (d-1) (1-p)}E_{2^{j-1}}(f)_p^p\Bigr)^{\f1p}\\
&\leq C \|f\|_p + C \Bigl(\sum_{k=1}^\infty k^{ (d-1) (1-p)-1}E_{k}(f)_p^p\Bigr)^{\f1p}<\infty,\end{align*}
where the second step uses the Nikolskii inequality.
\end{proof}

We are now in a position to prove Theorem \ref{thm-5-1-TD}.\\

{\it Proof of Theorem \ref{thm-5-1-TD}.}  Assume that $2^{m-1}\leq n <2^m$. Let $f_{2^j}\in\Pi_{2^j}^d$ be such that
$\|f-f_{2^j}\|_{p}=E_{2^j}(f)_p$.  Set $g_0=g_0$, and $g_j=f_{2^j}-f_{2^{j-1}}\in\Pi_{2^j}^d$ for $j\ge 1$. For each $g_j$, let $\s_j:=\sum_{k=0}^\infty Y_{j,k}, \   \ Y_{j,k}\in \HH_k^d$ be the Fourier-Laplace series built from Lemma \ref{lem-5-4-TD}. Thus, by Lemma \ref{lem-5-4-TD},
$$S_{2^j} \s_j =\sum_{k=0}^{2^j}Y_{j,k}=g_j,$$
and for any $n\ge 2^j$,
$$\|V_n \s_j\|_p \leq C \Bigl( \f {2^j}n\Bigr)^{(d-1)(\f 1p-1)}\|g_j\|_p\leq C\Bigl( \f {2^j}n\Bigr)^{(d-1)(\f 1p-1)}E_{2^{j-1}}(f)_p.$$

Now define
$$\sa:=\sum_{k=0}^\infty Y_k (x):=f_0(x)-\sum_{j=1}^\infty \sum_{k>2^j} Y_{j,k}(x),$$
where $$Y_k(x):=-\sum_{0\leq j\leq \log_2 k} Y_{j,k}(x),\   \   \text{for $k>1$}.$$
Then
\begin{align*}
V_n\s&=f_0-\sum_{j=1}^m V_n \Bigl(\sum_{2^j<k\leq 2n} Y_{j,k}\Bigr)=f_0-\sum_{j=1}^m V_n \Bigl(\sum_{k=0}^{2n} Y_{j,k}\Bigr)+\sum_{j=1}^m V_n \Bigl(\sum_{k=0}^{2^j} Y_{j,k}\Bigr)\\
&=f_0-\sum_{j=1}^m V_n(\s_j) +\sum_{j=1}^m V_n(g_j)=f_0-\sum_{j=1}^m V_n(\s_j) +\sum_{j=1}^{m-1} g_j +V_n(g_m) \\
&=f_{m-1} +V_n (g_m) -\sum_{j=1}^m V_n (\s_j).
\end{align*}
It follows that
\begin{align*}
\|f-V_n\sa\|_p^p &\leq \|f-f_{m-1}\|_p^p +\|V_n(g_m)\|_p^p +\sum_{j=1}^{m} \|V_n (\s_j)\|_p^p\\
&\leq C E_{2^{m-1}}(f)_p^p +C\sum_{j=1}^{m}  2^{(j-m)(1-p)(d-1)}\|g_j\|_p^p\\
&\leq C 2^{-m(1-p)(d-1)}\sum_{j=1}^m 2^{j(1-p)(d-1)} E_{2^j}(f)_p^p.
\end{align*}
Thus,
$$\|f-V_n\sa\|_p \leq C n^{-(d-1)(\f 1p-1)} \Bigl( \sum_{k=1}^n k^{d-2-(d-1)p} E_k (f)_p^p\Bigr)^{\f1p}.$$
This completes the proof of the first part.

To show the second part, we first observe that by Lemma \ref{lem-5-5-TD}, if $f$ satisfies \eqref{5-3-TD}, then it must be in $L^1(\sph)$, and hence $V_n f$ is defined. It follows that
\begin{align*}
&\|f-V_n f\|_p^p \leq \|f-f_{2^{m-1}}\|_p^p + \|V_n (f_{2^{m-1}})-V_n f\|_p^p
 \\&
 \leq E_{2^{m-1}} (f)_p^p +\sum_{j=m}^\infty \|V_n (f_{2^j}-f_{2^{j-1}})\|_p^p
\leq E_{2^{m-1}} (f)_p^p
\\
&+C \sum_{j=m}^\infty (2^j+n)^{(d-1)(1-p)} \int_{\sph}\int_{\sph} |K_n(x\cdot y) | ^p |f_{2^j}(y)-f_{2^{j-1}}(y)|^p\, d\s(y)\, d\s(x)\\
&\leq E_{2^{m-1}} (f)_p^p +C \sum_{j=m}^\infty n^{(d-1)(p-1)} 2^{j(d-1)(1-p)}E_{2^{j-1}} (f)_p^p\\
&\leq C n^{-(d-1)(1-p)}\sum_{k=[n/2]}^\infty k^{(d-1)(1-p)-1} E_k (f)_p^p.
\end{align*}
To complete the proof, we just need to observe that
$$f-V_n f=(f-P)-V_n(f-P),\   \    \   \forall P\in\Pi_n^d.$$\hb

Let us present a similar result for the   moduli of continuity on the sphere $\o (f, t)_p$ introduced by Ditzian (\cite{di, di1}). 
Let  $SO(d)$ denote the group of all  $d\times d$ orthogonal matrices.
Given $t\in (0,\pi)$, we denote  by  $O_t$ the class of matrices $\rho\in SO(d)$ such that
 $(\rho x \cdot x )\ge \cos t$ for all $x\in \sph$.
The first order  modulus of continuity is then defined  by
$$\omega(f,t)_p = \sup\limits_{\rho \in O_t} \|\Delta_\rho f\|_{L^p(\sph)},$$
 where $\Delta_\rho f = f(\rho x)-f(x)$.

 Using Theorem \ref{thm-5-1-TD}, and the Jackson inequality $E_n(f)_p\leq C_p \o(f, n^{-1})_p$ for $0<p<1$ proved in
  \cite[Theorem 4.1]{dd}, we deduce the following corollary:

\begin{cor}
If $f\in L^p(\sph)$ with $0<p<1$ then
there exists a Fourier-Laplace series $\sa$ on the sphere which is summable to $f$ by the generalized de la Vall\'ee Poussin means with the rate
$$ \|f-V_n \sa\|_p \leq C_{p,d} n^{-(d-1)(\f 1p-1)} \Bigl( \int_{n^{-1}}^{\pi} t^{(d-1)p -d} \o (f, t)_p^p\, dt \Br)^{\f1p}.$$
\end{cor}

\section{
Sobolev-type embedding with weights}
In this section we study an embedding theorem for weighted Besov spaces.
Let $E_n(f)_{p,w}$ be  the
best approximation of $f \in L_{p,w}$ by spherical polynomials of degree at most $n$ in the
$L_{p,w}$-metric. Given $0 < p \le\infty, \nu > 0$ and $0 < \tau\le \infty$,  the weighted Besov space $B^\nu_{\tau}(L_{p,w})$ is the collection of all functions
$f\in L_{p,w}$ with finite quasi-norm
$$\|f\|_{B^\nu_{\tau}(L_{p,w})}=
\|f\|_{p,w}+\Big(\sum_{j=0}^\infty2^{j\nu\tau}
E_{2^j}(f)_{p,w}^\tau\Big)^{1/\tau},$$
with the usual change when $\tau=\infty$.


The following Sobolev-type embedding result for the Besov space on $\mathbb{R}^d$ with the limiting smoothness
parameter is well known: $B^r_{q}\big(L_{p}(\mathbb{R}^d)\big)\hookrightarrow L^{q}(\mathbb{R}^d)$,
$r=d\bigl(\frac1p-\frac1q\bigr)>0$ (see, e.g., \cite[(8.2)]{peetre}).

For functions on $\sph$, it was shown in \cite[Th. 2.5]{DW} that if $0<p<q\le\infty$ and $w$ is doubling, then for
 $\nu>s_w(\frac1p-\frac1q)$ one has 
$B^\nu_q(L_{p,w})\subset L_{q,w}$. In the unweighted case this result was obtained in \cite[Cor. 4]{hesse}.
 Our next theorem extends the previous results for the limiting smoothness parameter. 


\begin{thm}\label{thm-last} If $0<p<q < \infty$ and $w$ is doubling, then for $\nu:=s_w(\frac1p-\frac1q)$ we have
$B^\nu_q(L_{p,w})\subset L_{q,w}$ and
$$\|f\|_{q,w}\le C \|f\|_{B^\nu_q(L_{p,w})}$$
for all $f\in B^\nu_q(L_{p,w})$. Furthermore, if
$0<p<\infty$ and
 $\nu=\frac{s_w}{p}$, then each function $f\in B^\nu_\infty(L_{p,w})$ can be identified
with a continuous function on $\sph$.
\end{thm}

For the proof of \eqref{thm-last}, we need the following lemma, which follows directly  from \cite[Lemma 4.2]{dt}, and Lemma \ref{lem-2-6-TD}.

\begin{lem}\label{lem-9-2} Assume that  $0<p<q\leq  \infty$ and    $f\in L_{p,w}$.  Let $\{f_{2^n}\}_{n=1}^{\infty}$ be a sequence of spherical polynomials such that $f_{2^n}\in\Pi_{2^n}^d$, and $\|f-f_{2^n}\|_{p,w}\leq C_1 E_{2^n}(f)_{p,w}$ for each $n\in\NN$ and some positive constant $C_1$.  Then for any $N\in\NN$,
$$ \|\sum_{n=1}^N (f_{2^n}-f_{2^{n-1}})\|_{q,w} \leq C_{p,q,w} \Bigl( \sum_{n=1}^N \bigl(2^{ns_w(\f 1p-\f1q)} E_{2^n}(f)_{p,w}\bigr)^{q_1}\Bigr)^{\f1{q_1}}.$$
where
$$q_1:=\begin{cases} q, \  \  \text{if $0<q<\infty$,}\\
1, \   \  \text{if $q=\infty$}.\end{cases}
$$
\end{lem}

 We point out that Lemma 4.2 of \cite{dt} applies to a more general setting, where  the Nikolskii type inequality is applicable.

 Now we are in a position to show Theorem \ref{thm-last}.\\

{\it Proof of Theorem \ref{thm-last}.} 
The proof runs along the same lines as that in \cite[Th. 4.1]{dt}, but is different from  those in  \cite[Cor. 4]{hesse} and \cite[Th.2.5]{DW}. Let
  $f_{2^j}\in\Pi_{2^j}^d$ be such that $E_{2^j}(f)_{p,w}=\|f-f_{2^j}\|_{p,w}$ for $j\ge 0$. Using Lemma \ref{lem-9-2},
 we obtain
  $$\Big\|\sum\limits_{j=1}^N \big(f_{2^j}- f_{2^{j-1}}\big)\Big\|_{q,w}\le C \left(
  \sum\limits_{j=1}^N \Big(
  (2^{js_w(\f 1p-\f1q)} E_{2^j}(f)_{p,w}
  \Big)^{q_1}
  \right)^\frac{1}{q_1},\   \  \forall N\in\NN.$$
  Since
  $$f=f_1+\sum_{j=1}^\infty (f_{2^{j}}-f_{2^{j-1}}),$$
  with  the series converging  in $L_{p,w}$-metric, it follows by Fatou's lemma and equivalence of different metrics on the finite-dimensional linear space $\Pi_2^d$ that for $q<\infty$
  \begin{align*}
  \big\|f\big\|_{q,w}&\le C_q\|f_1\|_{q,w} + C_q\liminf_{N\to\infty}\Big\|\sum\limits_{j=1}^N \big(f_{2^j}- f_{2^{j-1}}\big)\Big\|_{q,w}\\
  &\leq C  \big\|f_1\big\|_{p,w} + C
\left(  \sum\limits_{j=1}^\infty \Big(
  2^{j\nu} E_{2^j}(f)_{p,w}  \Big)^{q}
  \right)^\frac{1}{q}\\
  &\leq C  \big\|f\big\|_{p,w} + C
\left(  \sum\limits_{j=1}^\infty \Big(
  2^{j\nu} E_{2^j}(f)_{p,w}  \Big)^{q}
  \right)^\frac{1}{q}\sim \|f\|_{B_q^\nu (L_{p,w})},\end{align*}
  where $\nu=s_w(\f 1p-\f1q)$. A similar argument  works equally well for the case $q=\infty$.
\hb

 Given a doubling weight $w$, using \eqref{2-5-0-ch3}, it is easily seen that
 \begin{equation}\label{9-1-eq}\min_{x\in\sph}w(B(x, n^{-1}))\ge c_{w} n^{-s_w},\  \  \forall n\in\NN,\end{equation}
where $c_w>0$ is independent of $n$ and $x$.
We shall show that the index $\nu:=s_w(\f 1p-\f1q)$ in Theorem \ref{thm-last} is sharp  under the following additional assumption on the doubling weight $w$:
\begin{equation}\label{9-2-eq}
\min_{x\in\sph} w(B(x, n^{-1})) \leq c_w' n^{-s_w},\   \  n=1,2,\cdots.
\end{equation}
More precisely, we shall prove that under the condition of \eqref{9-2-eq},   given any $0<\nu'<\nu:=s_w(\f 1p-\f 1q)$, there exists a function $f$ which satisfies  $f\in B_\tau^{\nu'}(L_{p,w})$ for all $\tau>0$, but  $f\notin L_{q,w}$.
Indeed, conditions \eqref{9-1-eq} and  \eqref{9-2-eq} imply that there exists a sequence of points $y_n \in \sph$ such that
 \begin{equation}\label{9-3-eq}w(B(y_n, n^{-1})) \sim  n^{-s_w},\   \  n=1,2,\cdots.\end{equation}
On the other hand, by Lemma 4.6 of \cite{Dai1},  there exists a sequence of positive spherical polynomials $f_n$ such that
$f_n \in \Pi_{n}^d$ and
$$ f_n(x)\sim (1+n \rho(x, y_n))^{-\ell},\  \  \forall x\in \sph,$$
where $\ell$ is any given positive number greater than $ \f 1p (s_w+d)$.
A straightforward calculation, using \eqref{9-3-eq}  and \eqref{2-4-ch3}, then shows that
$$\|f_n\|_{p_1, w}\sim \|f_n\|_{p_1, w_n} \sim n^{-s_w/p_1},\   \  \forall p_1\ge p. $$
Let
\begin{equation} \label{emb-exa}f= \sum_{n=1}^{\infty} 2^{n s_w/q} 2^{n\va} f_{2^n},\end{equation}
where $\va$ is a positive constant satisfying $0<\va<\nu-\nu'$.
Then, with $\t:=\min\{p,1\}$, we have
$$E_{2^n }(f)_{p,w}\leq \Bigl(\sum_{k\ge n }2^{{n\t s_w}/{q}} 2^{n\va \t}\|f_{2^n}\|_{p,w}^\t\Bigr)^{\f 1\t}\leq C 2^{-n s_w(\f 1p-\f1q)+n\va }=C2^{-n \nu} 2^{n\va }.$$Thus, for any $\tau>0$,
 $$\|f\|_{B_\tau^{\nu'} (L_{p,w})} \leq C \Bigl(\sum_{n=1}^\infty
 2^{n \nu' \tau}2^{-n\tau \nu} 2^{n\tau\va } \Bigr)^{\f1{\tau}}<\infty.$$
 In particular, this implies that the series \eqref{emb-exa}  converges in $L_{p,w}$-metric. Next, we show that $f\notin L_{q,w}$. To see this, we note that each term $f_{2^n}$ in the series on the right hand side of \eqref{emb-exa}  is nonnegative, thus, by the monotone convergence theorem, $f\in L_{q,w}$ if any only if the series on the right hand side of \eqref{emb-exa} converges in $L_{q,w}$-metric, but this is impossible, since
 $$2^{n s_w/q}2^{n\va} \|f_{2^n}\|_{q,w} \sim 2^{n \va }\to \infty\   \ \text{as $n\to \infty$}.$$
 This completes the proof.

We conclude this section with the following remark.

\begin{rem}
It is very easy to verify  that all weights of the form \eqref{weight} satisfy the condition \eqref{9-2-eq}. In general, one can show that if a doubling weight $w$ satisfies the condition $$ \min_{x\in\sph} w(B(x, n^{-1})) \sim n^{-\xi},\   \ n=1,2,\cdots,$$
for some $\xi>0$,
then the Nikolski inequality \eqref{Nil} and Theorem \ref{thm-last} with $s_w=\xi$ hold, and in both cases, the index $\nu:=\xi(\f 1p-\f1q)$ is sharp.
\end{rem}

\section{Appendix: Proof of Theorem  \ref{prop-1-2-ch3} }

The main purpose in this section is to prove
 Theorem \ref{prop-1-2-ch3}. The proof   relies on the following
   two   lemmas. Let us  recall that $\alpha\ge \beta\ge -\frac 12$.

\begin{lem} If $k$ is a nonnegative integer, and $\t\in [0,\pi]$, then
\begin{equation}\label{3-3}
\f 1{ 2k +\al+\be +2} E_k ^{(\al+1, \be)} (\cos\ta )=\sum_{j=0}^k
E_j^{(\al, \be)} (\cos\ta ),\    \    \   \ta\in [0,\pi]
\end{equation}
and
\begin{equation}\label{3-4}
| E_k^{(\al,\be)} (\cos \ta)|\leq \begin{cases} C k^{2\al+1},&\ \
\text{if $0\leq \ta\leq k^{-1}$,}\\
C k^{\al+\f12} \ta^{-\al-\f12} (\pi-\ta)^{-\be-\f12},&\   \
\text{if
$k^{-1}\leq \ta \leq \pi-k^{-1}$},\\
C k^{\al+\be+1},&\   \  \text{if $\pi-k^{-1}<\ta \leq
\pi$}.\end{cases}\end{equation}
If, in addition, $\t\in [0, (2k)^{-1}]$, then
\begin{equation}\label{lowerbound}
E_k^{(\a,\b)}(\cos\t) \ge \f 12 E_k^{(\a,\b)}(1)\sim k^{2\a+1}.
\end{equation}

\end{lem}

\begin{proof}
Equation ($\ref{3-3}$) follows directly  by \cite[p. 257, (9.4.3)]{Sz}
while inequality ($\ref{3-4}$) is a simple consequence of
($\ref{2-1}$) and \cite[(7.32.5), (4.1.3)]{Sz} and the following fact:
\begin{equation}\label{3-5} \f {\Ga(x+a)}{\Ga(x)} = x^a+O(x^{a-1})\    \    \
\text{as $x\to \infty$},   \    \   a\in \mathbb{R}.\end{equation}
Finally, \eqref{lowerbound} follows directly from Bernstein's inequality:
$$|E_n^{(\a,\b)}(\cos\t)-E_n^{(\a,\b)}(1)|\leq n \t\|E_n^{(\a,\b)}\|_{\infty}=n\t E_n^{(\a,\b)}(1).$$
\end{proof}

\begin{lem}\label{lem-3-3-DT} If $r>0$, and $\t\in [0,\pi]$, then
\begin{equation}|G_{n,r}(\cos\t)|\leq c n^{2\a+2+r} (1+n\t)^{-(2\a+2+r)}.
\end{equation}
\end{lem}
\begin{proof} Assume that $2^{m-1}\leq n<2^m$, and set  $\psi(x)=\eta(x/2)-\eta(x)$. Since
$$\eta(\f kn) =\eta(\f k{2^{m+1}})\eta(\f kn) =\sum_{j=0}^m (\eta(\f k{2^{j+1}})-\eta(\f k{2^j}))\eta(\f kn)+\eta(k) \eta(\f kn),$$
it follows that
$$G_{n,r}(x)=\sum_{j=0}^{m+2}F_j(x),$$
where
$F_0(x)=(2+\a+\b)^{\f r2} E_1^{(\a,\b)}(x)$, and
$$ F_j(x)=\sum_{k=2^{j}}^{2^{j+2}}\eta( \f k n)\psi(\f k{2^j}) (k(k+\a+\b+1))^{\f r2} E_k^{(\a,\b)}(x).$$
Using Lemma \ref{lem-1-1-sec1} with $N=2^j$ and $\vi(x)=\eta( \f {2^jx} n)\psi(x) (x(x+2^{-j}(\a+\b+1)))^{\f r2}$, we obtain
\begin{equation}|F_j(\cos\t)|\leq C 2^{j(r+2\a+2)}(1+2^j\t)^{-\ell},\   \   \forall \ell>0.\end{equation}
Thus, choosing $\ell>r+2\a+2$, we obtain
\begin{align*}
|G_{n,r}(\cos\t)|&\leq c \sum_{j=0}^m 2^{j(r+2\a+2)}(1+2^j\t)^{-\ell}\\
&\leq c\sum_{0\leq j\leq \min \{m, \log_2\t^{-1}\} } 2^{j(r+2\a+2)}+
\sum_{\min \{m, \log_2\t^{-1}\}< j \leq m} \t^{-\ell} 2^{j ( r+2\a+2-\ell)}\\
&\leq c n^{2\a+2+r}(1+n\t)^{-2\a-2-r}.
\end{align*}
\end{proof}

Now we are in a position to prove Theorem \ref{prop-1-2-ch3}.\\

{\it Proof of Theorem \ref{prop-1-2-ch3}.}
We first show that
\begin{equation}\label{3-6}
\|G_{n}\|_{p,\a,\b} \sim n^{(2\a+2)(1-\f 1p)}.
\end{equation}
Indeed,
the upper bound of \eqref{3-6} follows directly from Lemma \ref{lem-1-1-sec1} with $i=0$ and $\ell>\f {2\a+2}p$.
 On the other hand, using \eqref{lowerbound}, we deduce
$$ G_n(\cos\t)\ge \f 12 G_n(1) \sim n^{2\al+2},\   \  \t\in [0,(2n)^{-1}].$$
This, in particular,  implies
$$\|G_n\|_{p,\a,\b}\ge cn^{2\al+2}\Bigl(\int_0^{(2n)^{-1}} t^{2\a+1}\, dt\Bigr)^{\f1p}\sim n^{(2\a+2)(1-\f1p)},$$
which gives the desired lower estimate of \eqref{3-6}.
Thus, the proof of   \eqref{3-7} is reduced to showing that
\begin{equation}\label{3-7-equv}
\|G_{n,r} \|_{p, \a,\b} \sim \begin{cases}  n^{r-(2\a+2)(\f1p-1)},\    \   &\text{if $r>(2\a+2)(\f 1p-1)$},\\
1,\   \   &\text{if $r<(2\a+2)(\f 1p-1)$, and $r\notin \NN$},\\
\log^{\f 1p} n,\    \   & \text{if $r=(2\a+2)(\f 1p-1)$ and $r\notin\NN$.}\end{cases}
\end{equation}

The upper estimates  of \eqref{3-7-equv} follows directly from Lemma \ref{lem-3-3-DT}, while the proof of the desired  lower estimates  for the case of $r>(2\a+2)(\f 1p-1)$ can be done almost identically as that of \eqref{3-6}.

The lower estimates of  \eqref{3-7-equv} for the remaining cases  can be deduced directly from the following  crucial  lemma, which is  of independent interest.

\begin{lem}\label{lem-3-4-TD}  Let $r>0$, and assume that  $r$ is not an even integer if  $\a+\b+1>0$, and $r$ is not an integer if $\a+\b+1=0$. Then  for any  $\t \in [An^{-1}, \va]$,
\begin{equation}\label{3-11-TD} |G_{n,r} (\cos \t)| \sim  \t^{-(2\a +2+r)},\   \     \   \forall n \ge A\va^{-1},\end{equation}
where $A$ and $\va$ denote a sufficiently large and, respectively, small positive constants, both depending only on $\a$ and $r$.
\end{lem}

For the proof of  Lemma \ref{lem-3-4-TD}, we need some  well-known results  for  the Ces\`aro  kernels  of the Jacobi polynomial  expansions,  defined as follows:
$$ S_{n}^{\d, (\a,\b)}(x)=\f 1{A_n^\d}\sum_{k=0}^n A_{n-k}^\d E_k^{(\a,\b)}(x),\   \   \d>0,\    \  x\in[-1,1].$$

\begin{lem}\label{lem-3-5-DT}
(i) If $\d\ge \a+\f 32$ and $\t\in [0, \f \pi 2]$, then
\begin{equation}\label{Ces} |S_n^{\d, (\a,\b)} (\cos\t)|\leq C n^{2\a+2} (1+n\t)^{-2\a-3}.\end{equation}

(ii) If $\d\ge \a+\b+2$, then the Ces\`aro $(C,\d)$-kernels $S_{n}^{\d, (\a,\b)}$ are positive on $[-1,1]$; that is,
\begin{equation}\label{pos} S_n^{\d, (\a,\b)}(x)\ge 0,\    \    x\in [-1,1].\end{equation}

\end{lem}
The results of Lemma \ref{lem-3-5-DT} are well known. Indeed,
\eqref{Ces} can be found in \cite[Theorem 2.1]{BC}, whereas \eqref{pos} was proved in \cite{AG} and \cite[(4.13)]{G}.

In summary, we have reduced the proof of \eqref{3-7-equv} to showing Lemma \ref{lem-3-4-TD}. The proof of this lemma is given as follows: \\

{\it Proof of Lemma \ref{lem-3-4-TD}.}  The upper estimate of \eqref{3-11-TD}    has already been given in Lemma \ref{lem-3-3-DT}. So we only need to show the lower estimate of \eqref{3-11-TD}.

  For simplicity, we assume that $\a+\b+1>0$. The proof below with a slight modification works equally well for the case when $\a+\b+1=0$ and $r$ is not an integer.  Let  $\ell$ be the smallest positive integer bigger than $\a+\b+r+2$.  Define $\ell+1$ functions $a_{n,r,j}: [0,\infty)\to\RR$, $j=0,1,\cdots,\ell$ iteratively  by
\begin{align*}
a_{n,r,0}(s)&= (2s +\a+\b+1) (s(s+\a+\b+1))^{\f r2} \eta(\f{s}n),\\
a_{n,r,j+1}(s)& =\f{ a_{n,r,j}(s)}{ 2s+\a+\b+j+1}-\f{
a_{n,,r,j}(s+1)}{ 2s+\a+\b+j+3}\\
&=-\int_0^{1} \f d{dt} \Bigl[ \f{a_{n,r,j} (s+t)}{2(t+s)+\a+\b+j+1}\Bigr]\, dt,\   \  j=0,\cdots,\ell-1.\end{align*}
Since  $\eta$ equals $1$ on $[0,1]$ and $\a+\b+1>0$, using induction on $j$, it is easily seen that for $0\leq j\leq \ell$,
\begin{equation} \label{asy}a_{n,r,j}(s) = \gamma_{r,j} s^{r+1-2j}+ s^{r-2j}  g_j(s^{-1}),\   \   1\leq s\leq n-j,\end{equation}
 some functions $g_j\in C^\infty [0,\infty)$,
where
$\g_{r,0} =2$, and $\g_{r,j}:=2^{1-j} (-1)^jr(r-2)\cdots (r-2j+2)$ for $j\ge 1$.   Moreover, a similar argument shows that
$$ |a_{n,r,j} (s)|\leq c _j (s+1)^{r+1-2j},\   \  \forall s\ge 0.$$
Note that the constant $\g_{r,j}$ will never be zero if $r$ is not an even integer.

Next, using  \eqref{3-3} and  summation by parts $\ell$
times, we  obtain
\begin{equation}\label{1-5-0-DT} G_{n,r}(t) =c
\sum_{k=0}^{2n} \f{a_{n,r,\ell}(k)}{2k+\a+\b+\ell+1}  E_k^{(\a+\ell, \b)} (t),\end{equation}
 for some nonzero  constant $c$ depending only on  $\a$ and $\b$.
Let $v$ be the smallest positive integer greater than $\a+\b+2$.  Using summation by parts $v+1$ times, we deduce from \eqref{1-5-0-DT}  that
 \begin{equation}\label{1-5-0-DT-1} G_{n,r}(t) =C
\sum_{k=0}^{2n} \Bigl[\overrightarrow{\tr}^{v+1} \f{a_{n,r,\ell}(k)}{2k+\a+\b+\ell+1} \Bigr] A_k^v S_k^{v, (\a+\ell,\b)}(t),\end{equation}
where $\overrightarrow{\tr} \mu_k=\mu_k-\mu_{k+1}$ and $\overrightarrow{\tr}^{i+1}=\overrightarrow{\tr}\overrightarrow{\tr}^i$. Setting
$$ \vi(s)=\f{a_{n,r,\ell}(s)}{2s+\a+\b+\ell+1},$$
and using \eqref{asy}, we have, for $1\leq s\leq n -\ell-v-1$,
\begin{equation}
\vi^{(v+1)} (s) = c_{v,\ell}s^{r-2\ell-v-1}  + s^{r-2\ell-v-2} g(s^{-1}),
\end{equation}
where $g$ is a $C^\infty$-function on $[0,\infty)$, and
$$ c_{v,\ell}=2^{-1}\g_{r,\ell}(-1)^{v+1}(r-2\ell+1)(r-2\ell)\cdots (r-2\ell-v).$$
It then follows that for $1\leq k\leq n-2\ell$,
\begin{align}\overrightarrow{\tr}^{v+1}\vi(k)&=(-1)^{v+1}\int_{[0,1]^{v+1}} \vi^{(v+1)} (k+t_1+\cdots+t_{v+1})\, dt_1\cdots dt_{v+1}\notag\\
&=(-1)^{v+1}c_{v,\ell} k^{ r-2\ell-v-1}+O\Bl(k^{r-2\ell-v-2}\Br).\label{diff}\end{align}
  Since $r$ is not an even integer, and $2\ell>r+1$, the constant $c_{v,\ell}$ is not zero. For $n-2\ell\leq k\leq 2n$, we have the following easy estimate
 \begin{equation}\label{3-19-TD}
 |\overrightarrow{\tr}^{v+1}\vi(k)|\leq c k^{ r-2\ell-v-1}.
 \end{equation}

 Thus, using \eqref{3-19-TD} and \eqref{diff},  we may rewrite \eqref{1-5-0-DT-1} in the form
 $$G_{n,r}(\cos\t)=c \sum_{k=1}^{n-2\ell} k^{r-2\ell-v-1} A_k^v S_k^{v, (\a+\ell,\b)} (\cos\t)+ R_{n,1}(\t)+R_{n,2}(\t),$$
 where $c\neq 0$, and
 \begin{align*} |R_{n,1}(\t)|&\leq C\sum_{k=[n/2]}^{2n} k^{r-2\ell-1}|S_k^{v, (\a+\ell,\b)} (\cos\t)|,\\
 |R_{n,2}(\t)|&\leq C\sum_{k=1}^{n} k^{r-2\ell-2}|S_k^{v, (\a+\ell,\b)} (\cos\t)|.\end{align*}
 Since the Ces\`aro kernels $S_k^{v, (\a+\ell,\b)} (\cos\t)$ are positive by (\ref{pos}), it follows that for $\t\in [An^{-1},\va]$,
 \begin{align*}
 \sum_{k=1}^{n-2\ell}& k^{r-2\ell-v-1} A_k^v S_k^{v, (\a+\ell,\b)} (\cos\t)\ge \sum_{1\leq k\leq 2^{-1}\t^{-1}} k^{r-2\ell-v-1} A_k^v S_k^{v, (\a+\ell,\b)} (\cos\t)\\
 &\ge c \sum_{1\leq k\leq 2^{-1}\t^{-1}} k^{ r+2\a+1}\ge c_1\t^{-(r+2\a+2)},
 \end{align*}
 where we have used the positivity of $S_k^{v, (\a+\ell,\b)}$ in the first step, and  \eqref{lowerbound} in the second step.

 To estimate the reminder term $R_{n,1}(\t)$, we use \eqref{Ces} and obtain \begin{align*}
 |R_{n,1}(\t)|&\leq C \sum_{k=[n/2]}^{2n} k^{r-2\ell-1} k^{-1} \t^{-(2\al+2\ell+3)}\leq  n^{r-2\ell-1} \t^{-(2\a+2\ell+3)}\\
 &=c (n\t)^{r-2\ell-1} \t^{-(2\a-r+2)}\leq c_2 A^{-(2\ell+1-r)}\t^{-(2\a-r+2)}\end{align*}provided that $n\t\ge A$.
 Similarly, using  \eqref{Ces}, we have
 \begin{align*}
 |R_{n,2}(\t)|&\leq c \sum_{1\leq k\leq \t^{-1}} k^{r-2\ell-2} k^{2(\a+\ell)+2} +c\sum_{\t^{-1}\leq k\leq 2n} k^{r-2\ell-2} k^{-1} \t^{-(2\a+2\ell+3)}\\
 & \leq c\t^{-(r+2\a+1)}+c \t^{-(r-2\ell-2)} \t^{-(2\a+2\ell+3)}\leq c\t^{-(r+2\a+1)}\leq c_3 \va \t^{-(r+2\a+2)}\end{align*}
 provided that $\t\leq \va$.

 Putting these together, we conclude that for $\t\in [n^{-1}A,\va]$,
 $$|G_{n,r}(\cos\t)|\ge \Bigl[c_1 -c_2A^{-(2\ell+1-r)}-c_3 \va \Bigr]\t^{-(r+2\a+2)}\ge c\t^{-(r+2\a+2)},$$
 provided that $A$ is large enough, and $\va$ is sufficiently small. This completes the proof.
  \hb

\vspace{5mm}


\bibliographystyle{amsalpha}

\begin{thebibliography}{BHW}

 \bi{arestov}
V. V. Arestov,
On integral inequalities for trigonometric polynomials and their derivatives,
{\it Math. USSR, Izv.} {\bf 18} (1982), 1--17; translation from {\it Izv. Akad. Nauk SSSR}, Ser. Mat. {\bf 45} (1981), 3--22.

 \bi{AG} R. Askey and G.  Gasper,  Positive Jacobi polynomial sums II. {\it  Amer. J. Math.} {\bf 98} (1976), no. 3, 709--737.

\bi{BL} E. Belinskii and E.  Liflyand,  Approximation properties in $L_p$, $0<p<1$, {\it Funct. Approx. Comment. Math.} {\bf 22} (1993), 189--199.

\bibitem {BC}   A. Bonami and   J. L.  Clerc, Sommes de Ces\`{a}ro et
multiplicateurs des d\`{e}veloppments en harmonique sph\'eriques,
   {\it Trans. Amer. Math. Soc.} {\bf 183} (1973), 223--263.

\bibitem{BD}
         G. Brown  and  F. Dai,
          Approximation of smooth functions on compact two-point
          homogeneous spaces,
          \textit{  J. Funct. Anal.}  {\bf 220} (2005), no. 2, 401--423.


   \bi{Dai5} F. Dai,  Strong convergence of spherical harmonic
expansions on $H\sp 1(S\sp {d-1})$, {\it  Constr. Approx.} {\bf
22} (2005), no. 3, 417--436.

   \bi{Dai1} F. Dai,  Multivariate polynomial inequalities with
respect to doubling weights and $A_ \infty$ weights, {\it  J.
Funct. Anal.} {\bf 235} (2006),  no. 1, 137--170.

 \bibitem{Dai2}
          F. Dai,  Jackson-type inequality for doubling weights on the sphere,
          {\it  Constr. Approx.} {\bf  24}  (2006),  91--112.


\bi{dd} F. Dai and Z. Ditzian,
Jackson theorem in $L_p,$ $0 < p < 1,$ for functions on
the sphere, {\it Journal of Approx. Theory},  {\bf 162} (2010), 382--391.

\bi{DW} F. Dai and H. P. Wang,  Optimal   cubature formulas
  in  weighted   Besov  spaces with $A_\infty$ weights
 on multivariate domains,
{\it Constr. Approx.}, {\bf 37} (2013), 167--194.



\bibitem{Di3}  Z. Ditzian, Fractional derivatives and best
approximation,\    {\it Acta Math. Hungar.},\    {\bf 81} (1998),
323--348.

\bibitem{di}
Z. Ditzian,
 A modulus of smoothness on the unit sphere, {\it J. Anal. Math.},  {\bf 79} (1999), 189--200.

\bibitem{di1}
Z. Ditzian, Jackson-type inequality on the sphere, {\it Acta Math. Hungar.},\    {\bf 102}  (2004), (1–2), 1--35.

\bibitem{dt}
 Z. Ditzian and S. Tikhonov,
{Ul'yanov and Nikol'skii-type inequalities}, {\it J. Approx. Theory}, {\bf 133} (2005), 1, 100--133.


\bibitem{er} T. Erd\'{e}lyi, Notes on inequalities with doubling weights, {\it J. Approx. Theory}, {\bf 100}  (1999), no. 1, 60--72.


  \bi{G}G.  Gasper,  Positive sums of the classical orthogonal polynomials, {\it  SIAM J. Math. Anal.} {\bf 8 }(1977), no. 3, 423--447.


  \bi{hesse}
K. Hesse, H. N. Mhaskar, and I. H. Sloan, Quadrature in Besov spaces on the Euclidean sphere,
{\it  J. of Complexity}, {\bf 23} (2007), no. 4-6, 528--552.

\bibitem{IPX}
          K. Ivanov, P. Petrushev, and Y. Xu,
          Sub-exponentially localized kernels and frames induced by orthogonal expansions,
          \textit{Math. Z.}, \textbf{264} (2010),  361--397.



\bibitem{kam}
A. I. Kamzolov, Bernstein's inequality for fractional derivatives of polynomials in spherical harmonics, {\it
Russian Mathematical Surveys},  {\bf 39}(2) 1984,  163; translation from {\it Uspekhi Mat. Nauk}, {\bf 39} (1984), no. 2 (236), 159--160.

\bibitem{lerner}
A. K. Lerner,  C. P\'{e}rez, A new characterization of the Muckenhoupt $A_p$ weights through an extension of the Lorentz-Shimogaki theorem,
{\it Indiana Univ. Math. J.}, {\bf  56} (2007), no. 6, 2697--2722.

  \bi{li}
P. I. Lizorkin, Estimates for trigonometric and the Bernstein inequality for fractional derivatives, {\it Izv. AN, Ser. Mat.}
{\bf 29} (1965) 109--126 (in Russian); translated in: {\it Am. Math. Soc., Transl.},  {\bf 77} (1968), 45--62.


  \bi{MT2}  G.  Mastroianni  and V. Totik,  Weighted polynomial
inequalities with doubling and $A_ \infty$ weights,  {\it Constr.
Approx.} {\bf 16} (2000), no. 1, 37--71.

  \bi{osw}
P. Oswald,
Rate of approximation by de la Vallee-Poussin means of trigonometric series in the metric of $L_p$ $(0<p<1)$, 
{\it Sov. J. Contemp. Math. Anal., Arm. Acad. Sci.} {\bf 18} (1983), no. 3, 63--78; translation from {\it Izv. Akad. Nauk Arm. SSR, Mat.} {\bf 18} (1983), no. 3, 230--245.

  \bi{peetre}
J.~Peetre, {Espaces d'interpolation et th\'{e}eor\`{e}me de Soboleff}, {\it Ann. Inst.
Fourier (Grenoble)}, \textbf{16} (1966), 279--317.


  \bi{ru1}
K. Runovskii and H.-J. Schmeisser, On some extensions of Bernstein's inequalities for trigonometric polynomials, {\it Funct. et Approx.} {\bf  29} (2004),
125--142.

  \bi{ru2}
K. Runovskii and H.-J. Schmeisser,
Inequalities of Calderón-Zygmund type for trigonometric polynomials,
{\it Georgian Math. J.}, {\bf 8} (2001), no. 1, 165--179.

\bibitem{Stein95}
        E. M. Stein,
       \textit{Harmonic Analysis: Real-Variable Methods, Orthogonality,
        and Oscillatroy Integrals},
        Princeton Univ. Press, Princeton, NJ, 1993.


  \bi{Sz}
G. Szeg\"{o}, {\it Orthogonal Polynomials}, Amer. Math. Soc., New York, 1967.

  \bi{tre}
W. Trebels, Multipliers for $(C, \alpha)$-bounded Fourier expansions in Banach spaces and
approximation theory, {\it Lecture Notes in Mathematics}, Springer, Vol. 329, 1973.

\bibitem {WL}
         K. Y. Wang     and  L. Q.  Li,
          \textit{Harmonic Analysis and Approximation on the unit Sphere},
          Science  Press, Beijing, 2000.


  \bi{wu}
J. Wu, Lower bounds for an integral involving fractional Laplacians and the generalized Navier-Stokes equations in Besov spaces,
{\it Commun. Math. Phys.}, {\bf 263} (2006), no. 3, 803--831.


        \bi{wu1}
J. Wu, Existence and uniqueness results for the 2-D dissipative quasi-geostrophic equation, {\it Nonlinear Anal., Theory Methods Appl.}
{\bf 67} (2007), no. 11, 3013--3036.





\end{thebibliography}

\end{document}